\documentclass[12pt]{amsart}
\usepackage[active]{srcltx}
\usepackage{calc,amssymb,amsthm,amsmath,amscd}
\usepackage{stmaryrd}
\usepackage[left=1.25in,top=1.2in,right=1.25in,bottom=1.2in]{geometry}

\RequirePackage[dvipsnames,usenames]{color}
\input{kmacros3.sty}
\input{xy}
\xyoption{all}
\usepackage{mabliautoref}
\usepackage{graphicx}

\renewcommand{\myH}{{\mathcal{H}}}

\usepackage{upgreek}
\newcommand{\mystar}{\diamondsuit}
\renewcommand{\m}{\fram}
\newcommand{\mytau}{{\uptau}}

\usepackage{setspace}
\usepackage{hyperref}
\usepackage{enumerate}
\usepackage{graphicx}

\usepackage[all,cmtip]{xy}
\usepackage{verbatim}

\theoremstyle{theorem}

\renewcommand{\O}{\mathcal O}

\DeclareMathOperator{\Spa}{Spa}

\begin{document}
\title[Perfectoid multiplier/test ideals]{Perfectoid multiplier/test ideals in regular rings and bounds on symbolic powers}
\author{Linquan Ma and Karl Schwede}
\address{Department of Mathematics\\ Purdue University\\ West Lafayette\\ IN, 47907, USA}
\email{ma326@purdue.edu}
\address{Department of Mathematics\\ University of Utah\\ Salt Lake City\\ UT 84112, USA}
\email{schwede@math.utah.edu}

\thanks{The first named author was supported in part by NSF Grant \#1836867/1600198 and NSF CAREER Grant DMS \#1252860/1501102.}
\thanks{The second named author was supported in part by the
  NSF FRG Grant DMS \#1265261/1501115 and NSF CAREER Grant DMS \#1252860/1501102.}

\maketitle

\begin{abstract}
Using perfectoid algebras we introduce a mixed characteristic analog of the multiplier ideal, respectively test ideal, from characteristic zero, respectively $p > 0$, in the case of a regular ambient ring.  We prove several properties about this ideal such as subadditivity.
%As an application, we derive a uniform bound on the growth of symbolic powers of radical ideals in regular rings analogous to results of Ein-Lazarsfeld-Smith and Hochster-Huneke.
We then use these techniques to derive a uniform bound on the growth of symbolic powers of radical ideals in all excellent regular rings.  The analogous result was shown in equal characteristic by Ein--Lazarsfeld--Smith and Hochster--Huneke.
\end{abstract}

\section{Introduction}
In this paper, we prove the following result on the uniform bound on the growth of symbolic powers of ideals.  %{\color{red}Recall that the $n$th symbolic power $Q^{(m)}$ of a prime ideal $Q \subseteq R$ is the set of ring elements that vanish to order $m$ in the localization $R_Q$.}
\begin{mainthm*}[\autoref{thm.SymbolicContainmentInGeneral}]
\label{thm.Main}
Suppose that $R$ is a Noetherian regular ring with reduced formal fibers (\itshape{e.g.} $R$ is excellent). If $ Q\subseteq R$ is a prime ideal of height $h$, then for all $m > 0$ we have
\[
Q^{(mh)} \subseteq Q^m.
\]
where $Q^{(mh)}$ is the $mh$ symbolic power of $Q$.\footnote{$Q^{(mh)}=Q^{mh}R_Q\cap R$, i.e., the elements of $R$ which vanish generically to order $mh$ at $Q$.}
\end{mainthm*}
%\begin{remark}
%Just as in \cite{EinLazSmithSymbolic} or \cite{HochsterHunekeComparisonOfSymbolic}, we actually obtain a more general statement about symbolic powers of radical ideals.
%\end{remark}

When $R$ is finite type over $\bC$, this result was proved as an application of multiplier ideals by Ein--Lazarsfeld--Smith \cite{EinLazSmithSymbolic}.  Shortly later, Hochster--Huneke \cite{HochsterHunekeComparisonOfSymbolic} used tight closure theory to prove the result when $R$ contains a field.  Also see \cite{SwansonLinearEquivalenceOfIdealTopologies} where it was first shown that there is a linear containment relation between symbolic and ordinary powers. Our contribution to the Main Theorem is the mixed characteristic case, which answers the question of Hochster--Huneke in \cite[Section 5]{HochsterHunekeComparisonOfSymbolic}.

Since regular local rings are UFDs, every height one prime $Q$ is principal and hence $Q^{(m)} = Q^m$. Thus the Main Theorem can be viewed as a strengthening and generalization of this classical fact to primes of higher codimension. Moreover, starting with the aforementioned results, the question of the growth of symbolic powers has been of central importance in commutative algebra and its applications to algebraic geometry over the past few decades, see for example \cite{HochsterHunekeFineBehaviorOfSymbolicPowers,HunekeKatzValidashtiUniformEquivalenceSymbolicAdic,BDRHKKSS,BocciHarbourneSymbolicPowers, DaoDeStefaniGrifoHunekeNunezSymbolicSurvey}.

The main ideas of our proof come from the recent solution of the direct summand conjecture and its derived variant \cite{AndreDirectsummandconjecture,BhattDirectsummandandDerivedvariant}: we introduce a mixed characteristic analog of the multiplier ideal or test ideal using perfectoid algebras, prove many properties of it, and finally (and analogously to the strategy of \cite{EinLazSmithSymbolic}, see also \cite{HaraACharacteristicPAnalogOfMultiplierIdealsAndApplications}) use those properties to deduce the Main Theorem above.

%\begin{theorem*}[\cite{EinLazSmithSymbolic,HochsterHunekeComparisonOfSymbolic}]
%Suppose that $R$ is a regular ring containing a field and that $Q \subseteq R$ is a prime ideal of height $h$, then for all $m > 0$
%\[
%Q^{(mh)} \subseteq Q^m
%\]
%where $Q^{(mh)}$ is the $mh$ symbolic power of $Q$.\footnote{$Q^{(mh)}=Q^{mh}R_Q\cap R$, i.e., the elements of $R$ which vanish generically to order $mh$ at $Q$.}
%\end{theorem*}

\subsection{Multiplier and test ideals}
%A driving force at the intersection of commutative algebra and algebraic geometry has been the relation between characteristic $p > 0$ commutative algebra (and in particular tight closure and test ideals) and higher dimensional algebraic geometry (and in particular multiplier ideals).  This relationship goes back at least to Fedder \cite{FedderFPureRat} and Mehta-Ramanathan \cite{MehtaRamanathanFrobeniusSplittingAndCohomologyVanishing} who both related Frobenius splittings and rational singularities.  The connection between these two fields inspired a number of parallel applications in characteristic $p > 0$ and characteristic zero, proved frequently using very different looking methods.  %One notable such application was by Ein-Lazarsfeld-Smith and Hochster-Huneke who proved the following
Suppose that $R$ is an equal characteristic regular domain satisfying mild geometric assumptions\footnote{For example, of essentially finite type over a field, or complete, or $F$-finite in characteristic $p > 0$.}.  Further suppose that $\fra \subseteq R$ is an ideal and $t \in\mathbb{R}_{\geq 0}$ a \emph{formal} exponent for $\fra$. In this setting we can form the test ideal $\tau(R, \fra^t)$ in characteristic $p > 0$ or the multiplier ideal $\mJ(R, \fra^t)$ in characteristic $0$.

This is an ideal of $R$ which measures the singularities of $V(\fra) \subseteq \Spec R$, scaled by $t$.  {\it Roughly speaking}, for relevant values of $t$,  the multiplier or test ideal of $(R, \fra^t)$ is smaller/deeper than that of $(R, \frb^t)$ if $V(\fra)$ has the same dimension as $V(\frb)$ and is more singular than $V(\frb)$. Crucially for the applications to symbolic powers, the multiplier or test ideal satisfies the following list of properties, see for example \cite{HochsterHunekeTC1,HaraYoshidaGeneralizationOfTightClosure,LazarsfeldPositivity2}.  We state them for the multiplier ideal $\mJ(R, \fra^t)$ but they also hold for the test ideal $\tau(R, \fra^t)$.
\begin{itemize}
\item[(A)] {\bf Basic containments:}   If $\fra \subseteq \frb$ is a containment of ideals, then
\label{itm.A}
\[
\mJ(R, \fra^t) \subseteq \mJ(R, \frb^t)
\]
and if $t < t'$ then
\[
\mJ(R, \fra^{t'}) \subseteq \mJ(R, \fra^{t}).
\]
\item[(B)] {\bf Unambiguity of exponent:}  For any positive integer $n$,
\[
\mJ(R, \fra^{tn}) = \mJ(R, (\fra^n)^{t}).
\]
\item[(C)] {\bf Not too small:} $\fra \subseteq \mJ(R, \fra)$.
\item[(D)] {\bf Not too big:} If $\fra$ is prime of height $h$, $\mJ(R, (\fra^{(lh)})^{{1\over l}}) \subseteq \fra$.
%\item{} {\bf Restriction:} If $0 \neq x\in R$ then
%\[
%\mJ(R/x, (\fra \cdot (R/xR))^t) \subseteq \mJ(R, \fra^t) \cdot (R/xR).
%\]
\item[(E)] {\bf Subadditivity:} If $\frb$ is another ideal and if $s \geq 0$ is another real number, then
\label{itm.E}
\[
\mJ(\fra^s \frb^t) \subseteq \mJ(\fra^s) \cdot \mJ(\frb^t).
\]
In particular we have
\[
\mJ(\fra^{tn}) \subseteq \mJ(\fra^t)^n.
\]
\end{itemize}

Combining these results, the application to the growth of symbolic powers follows from a clever asymptotic construction of multiplier ideals \cite{EinLazSmithSymbolic}, see also \cite{HaraACharacteristicPAnalogOfMultiplierIdealsAndApplications}.  We aim to do the same thing in mixed characteristic.

Very roughly, the multiplier ideal and test ideal of a regular local ring $(R, \fram)$ of dimension $d$ can be defined in the following way:
\[
\begin{array}{rl}
(\mJ/\tau)(R, \fra^t) = & \Ann_R \{\eta \in H^d_{\fram}(R) \;|\; \text{ $\eta$'s image in $H^d_{\fram}(B)$ is ``annihilated'' by $\fra^t$} \}.
\end{array}
\]
Here $B$ and ``annihilated'' are made precise as follows:
\begin{description}
\item[$B$ in characteristic zero]  $B = \myR \pi_* \O_Y$ where $\pi : Y \to \Spec R$ is a log resolution.\footnote{$\pi : Y \to \Spec R$ is proper birational, $Y$ is regular and $\fra \cdot \O_Y$ defines a SNC divisor.}  We define the $\eta$ which are ``annihilated'' as follows: write $\fra \cdot \O_Y = \O_Y(-G)$ and consider $\eta$ whose image in $\bH^d_{\fram}(\myR \pi_* \O_Y(\lfloor tG \rfloor))$ is zero.
\item[$B$ in characteristic $p > 0$]  $B = R^{1/p^{\infty}}$, the perfection of $R$.  We define the $\eta$ which are ``annihilated'' as follows:  those $\eta$ such that $c^{1/p^e} (\fra^{\lceil tp^e \rceil})^{1/p^e} \eta = 0 \in H^d_{\fram}(R^{1/p^{\infty}})$ for some $c \neq 0$ and all $e > 0$.
\end{description}

The real power of both the multiplier and test ideal (and their related circles of ideas) are the associated vanishing theorems that accompany them.  In characteristic zero this is Kawamata--Viehweg vanishing \cite{KawamataVanishing,ViehwegVanishingTheorems}, see for instance \cite{EsnaultViehwegLecturesOnVanishing}.  In characteristic $p>0$, Serre vanishing combined with the Frobenius morphism plays an analogous role.

Our goal in this article is to develop a theory of the test ideal in mixed characteristic regular local rings and prove it satisfies properties (A) through (E) above.

\subsection{Perfectoid test ideals} Now let us assume that $A$ is a complete regular local ring of mixed characteristic $(0,p)$. In this situation, for every fixed element $g\in A$, Andr\'{e} constructed an $A$-algebra $A_{\infty}$ that is an integral perfectoid algebra almost faithfully flat over $A$ mod powers of $p$ and such that $g\in A$ has a compatible system of $p$-power roots in $A_\infty$ (in this case, $g^{1/p^e}$ will be declared \emph{compatible}). This ingenious construction is crucial in the solution of the mixed characteristic case of the direct summand conjecture \cite{AndreDirectsummandconjecture,BhattDirectsummandandDerivedvariant} and the existence of big Cohen-Macaulay algebras \cite{AndreDirectsummandconjecture,HeitmannMaBigCohenMacaulayAlgebraVanishingofTor,ShimomotoIntegralperfectoidbigCMviaAndre,AndreWeaklyFunctorialBigCM}.

In this article, we iterate Andr\'{e}'s construction to obtain a huge extension $A\to A_{\infty}$ that is almost faithfully flat over $A$ and such that all elements of $A$ have a compatible systems of $p$-power roots in $A_\infty$. We will use this $A_\infty$ as the $B$ to replace $R^{1/p^{\infty}}$ in the definition of the test ideal in characteristic $p > 0$ (or as a replacement for the $\myR \pi_* \O_Y$ in the definition of the multiplier ideal in characteristic $0$).  Inspired by this, let $\fra\subseteq A$ be an ideal, we define the {\it perfectoid test ideal} of $(A, \fra^t)$ to be
\[
\mytau(A, \fra^t) = \Ann_A\{ \eta \in H^d_{\fram}(A) \;|\; \eta \text{ is ``almost'' annihilated by }\fra^t\}
\]
There are different ways to interpret this $\fra^t$ action in our setting, and at least in some proofs, it is convenient to define our analog of the test ideal with respect to a sequence of elements $\{f_1, \ldots, f_n\}$ that generate $\fra$.  See \autoref{sec.PerfMultiplierTestIdeals} for more details of these definitions.

We then show that the above perfectoid test ideals satisfy the analogs of properties (A) through (E) above in \autoref{prop.EasyContainments},  \autoref{proposition--test ideals of powers}, \autoref{proposition--test ideals of powers-GensVersion}, \autoref{proposition--test ideal contains the original ideal}, \autoref{lemma--asymptotic test ideal is contained in Q}, \autoref{theorem--subadditivity}.
Putting these together, and defining asymptotic perfectoid test ideals similar to how asymptotic multiplier ideals were introduced in \cite{EinLazSmithSymbolic}, we obtain our Main Theorem.

%\begin{mainapp*}[\autoref{thm.SymbolicContainmentInGeneral}]
%Suppose that $R$ is a Noetherian excellent regular ring. If $ Q\subseteq R$ is a prime ideal of height $h$, then for all $m > 0$ we have
%\[
%Q^{(mh)} \subseteq Q^m.
%\]
%\end{mainapp*}
%\begin{remark}
%Just as in \cite{EinLazSmithSymbolic} or \cite{HochsterHunekeComparisonOfSymbolic}, we actually obtain a more general statement about symbolic powers of radical %ideals.
%\end{remark}

Beyond this, it is also natural to compare our perfectoid test ideals with multiplier ideals in characteristic 0.  We obtain the following result, which is essentially a corollary of our proof of property (D) in mixed characteristic.

\begin{theorem*}[\autoref{prop.ComparisonWithMultiplier}, \autoref{thm.TauInMultiplier}]
Suppose that $(A, \fra^t)$ is a pair, where $A$ is a complete regular local ring of mixed characteristic $(0,p)$.  First suppose that $\pi : Y \to X = \Spec A$ is a proper birational map with $Y$ normal and such that $\fra \cdot \O_Y = \O_Y(-G)$.  Then
\[
\mytau(A, \fra^t) \subseteq \Gamma(Y, \O_Y(\lceil K_{Y/X}- tG \rceil))\footnote{See \autoref{def.RelCanonical} for a definition of $K_{Y/X}$ in this context.}
\]
where the object on the right would be the multiplier ideal if $Y$ is a log resolution.

Furthermore, since $A[1/p]$ has characteristic zero, we can form the multiplier ideal $\mJ(A[1/p], (\fra \cdot A[1/p])^t)$.  We have:
\[
\mytau(A, \fra^t) \cdot A[1/p] \subseteq  \mJ(A[1/p], (\fra \cdot A[1/p])^{t}).
\]
\end{theorem*}
We also expect that the characteristic zero statement is an equality, but we do not know how to show this.
%This difficulty is related to perhaps the biggest gap in the theory we have developed so far: how our perfectoid multiplier/test ideal behaves under localization, see \autoref{sec.FurtherQuestions} for additional discussion.
\vskip 12pt
\subsection*{Acknowledgements:}  The authors would like to thank Bhargav Bhatt, Raymond Heitmann, Kiran Kedlaya, Tiankai Liu, Stefan Patrikis, and Peter Scholze for valuable conversations.  We thank Rankeya Datta for comments on a previous draft.  Finally, we thank all the referees for numerous comments on previous versions -- their feedback has substantially improved the paper.

\section{Perfectoid algebras and Andr\'{e}'s construction}
\label{sec.PerfectoidAlgebrasAndAndre}

Throughout this paper we will use the language of (integral) perfectoid algebras and almost mathematics as in \cite{ScholzePerfectoidspaces}, \cite{GabberRameroAlmostringtheory}, \cite{BhattDirectsummandandDerivedvariant}, \cite{AndreWeaklyFunctorialBigCM}. We will work over a {\it fixed} perfectoid field $K=\widehat{\mathbb{Q}_p(p^{1/p^\infty})}$ and its ring of integers $K^\circ=\widehat{\mathbb{Z}_p[p^{1/p^\infty}]}$. We collect some definitions from \cite[Section 5]{ScholzePerfectoidspaces}, \cite[Section 1.4]{BhattDirectsummandandDerivedvariant}, \cite[Section 2.2]{AndreWeaklyFunctorialBigCM}.  Additionally, we use the notation $\myH^j(\bullet)$ to denote the $j$th cohomology of a complex.

A {\it perfectoid $K$-algebra} is a Banach $K$-algebra $R$ such that the subring of power-bounded elements $R^\circ\subseteq R$ is bounded and the Frobenius is surjective on $R^\circ/p$. A $K^\circ$-algebra $S$ is called {\it integral perfectoid} if it is $p$-adically complete, $p$-torsion free, and the Frobenius induces an isomorphism $S/p^{1/p}\xrightarrow{\sim} S/p$. If $R$ is a perfectoid $K$-algebra, then the subring of power-bounded elements $R^\circ$ is integral perfectoid, and if $S$ is integral perfectoid, then $S[1/p]$ perfectoid, see \cite[Theorem 5.2]{ScholzePerfectoidspaces}. Unless otherwise stated, almost mathematics in this paper will always be measured with respect to the ideal $(p^{1/p^\infty})\subseteq K^\circ$.

\begin{remark}
In \cite{BhattDirectsummandandDerivedvariant}, there is an extra condition in the definition of integral perfectoid algebra: one requires that $S=S_*:=\{x\in S[1/p] \hspace{0.5em}  | \hspace{0.5em}  p^{1/p^n} x\in S \text{ for all } n \}$. If we impose this extra condition then \cite[Theorem 5.2]{ScholzePerfectoidspaces} says the two categories are equivalent. In particular, $S_*=S[1/p]^\circ$ is integrally closed in $S[1/p]=S_*[1/p]$. Since passing from $S$ to $S_*$ is harmless for all our purposes (they are almost isomorphic to each other), we will assume that $S$ is integrally closed in $S[1/p]$ for all integral perfectoid algebras in the remainder of the article.
\end{remark}

We recall the following definitions.
\begin{itemize}
\item A map $A\to S$ such that $S$ is a $K^\circ$-algebra is {\it almost flat} if $\Tor_i^A(M, S)$ is almost zero (i.e., annihilated by $(p^{1/p^\infty})$) for all $A$-modules $M$ and all $i>0$. By taking syzygies and degree shifting, it suffices that $\Tor_1^A(M, S)$ is almost zero for all $A$-modules $M$.
\item A map $R\to S$ of $K^\circ$-algebras is {\it almost faithfully flat} if it almost flat, and such that if $M\otimes_RS$ is almost zero then $M$ is almost zero.
\end{itemize}

The goal of this section is to explain the following:

\begin{theorem}
\label{thm.ExistenceofAinfty}
Let $(A,\m)$ be a complete regular local ring of mixed characteristic $(0,p)$ and dimension $d$. Then there exists a map $A\to A_\infty$ to an integral perfectoid $K^\circ$-algebra $A_\infty$ such that:
\begin{enumerate}
  \item All elements of $A$ have a compatible system of $p$-power roots in $A_\infty$.
  \item $A\to A_\infty$ is almost flat. In particular, $A\hookrightarrow A_\infty$ is injective, and nonzero elements of $A$ are nonzerodivisors in $A_\infty$.
  \item If $M\otimes_AA_\infty$ is almost zero, then $M=0$.\footnote{We caution the reader that the term ``almost faithfully flat" is only defined when we consider maps of $K^\circ$-algebras while our base ring $A$ here is not defined over $K^\circ$ (e.g., saying $M$ is almost zero does not usually make sense here since $M$ is just an $A$-module). This is the reason we treat the properties $(b)$ and $(c)$ separately in the statement.}
\end{enumerate}
\end{theorem}

In order to prove \autoref{thm.ExistenceofAinfty}, we need the following fact about almost flat maps which we believe is known to experts (this is the ``almost" analog of \cite[Remark 4.31]{BhattMorrowScholzeIntegralPadicHodge}). We include an elementary proof since we cannot find a good reference.

\begin{lemma}
\label{lem.AlmostFlatComplete}
Let $A\to S$ be a map of $p$-adically complete and $p$-torsion free rings such that $A$ is Noetherian and $S$ is a $K^\circ$-algebra. If $A/p^k\to S/p^k$ is almost flat for all $k>0$, then $A\to S$ is almost flat.
\end{lemma}
\begin{proof}
We want show $\Tor_1^{A}(M, S)$ is almost zero for all $A$-modules $M$. We can assume that $M$ is finitely generated by taking direct limit. By considering $0\to \Gamma_{(p)}M\to M\to \overline{M}\to 0$, we only need to handle the case when $M$ is annihilated by $p^k$ for some $k$ and the case when $M$ is $p$-torsion free.

\subsubsection*{Case 1:}
If $M$ is annihilated by $p^k$, then we have
$$M\otimes^{\mathbf{L}}_AS\cong M\otimes^{\mathbf{L}}_{A/p^k}(A/p^k\otimes^{\mathbf{L}}_{A}S)\cong M\otimes^{\mathbf{L}}_{A/p^k}S/p^k$$
and hence $\Tor_1^{A}(M, S)\cong \Tor_1^{A/p^k}(M, S/p^k)$ is almost zero.

%write $0\to K\to F\to M\to 0$ where $F$ is a free $A/p^k$-module. We have
%\[
%0=\Tor_1^{A}(F, S)\to \Tor_1^{A}(M, S)\to K\otimes_A S\to F\otimes_A S
%\]
%where the leftmost $0$ is because $S$ is $p$-torsion free. From this sequence it is enough to prove $K\otimes_A S\to F\otimes_A S$ is an almost injection, which is clear since both $K$ and $F$ are $A/p^k$-modules and $S/p^k$ is almost flat over $A/p^k$.

\subsubsection*{Case 2:}
Now we assume that $M$ is $p$-torsion free. Let $F_\bullet\to M\to 0$ be a free resolution of $M$ with each term of $F_\bullet$ a finite free $A$-modules (since $A$ is Noetherian and $M$ is finitely generated).  We set $F_\bullet^{(m)}=F_\bullet\otimes A/p^m\cong M/p^m$ since $M$ is $p$-torsion free. Since $S$ is $p$-adically complete, we have $$\varprojlim_m (F_\bullet^{(m)}\otimes_A S)\cong \varprojlim_m(F_\bullet\otimes_A(S/p^m))\cong F_\bullet\otimes_AS.$$
Since $\{F_\bullet^{(m)}\otimes_A S\}_m$ forms a tower of chain complexes with surjective transition maps, we have a Milnor exact sequence \cite[Theorem 3.5.8]{WeibelHomological}:
\[
0\to {\varprojlim_m}^1 \myH^{-2}(F_\bullet^{(m)}\otimes_A S) \to \myH^{-1}(\varprojlim_m (F_\bullet^{(m)}\otimes_A S)) \to \varprojlim_m \myH^{-1}(F_\bullet^{(m)}\otimes_A S)\to 0.
\]
Since $S/p^m$ is almost flat over $A/p^m$, we know that
$$\myH^{-i}(F_\bullet^{(m)}\otimes_AS) \cong \myH^{-i}(F_\bullet^{(m)}\otimes_{A/p^m} (S/p^m))\cong \Tor_i^{A/p^m}(M/p^m, S/p^m)$$ is almost zero for every $i>0$. This together with the above discussion shows that $\Tor_1^A(M, S)=\myH^{-1}(F_\bullet\otimes_AS)=\myH^{-1}(\varprojlim_m (F_\bullet^{(m)}\otimes_A S))$ is almost zero.
\end{proof}

To prove \autoref{thm.ExistenceofAinfty}, we also need an ingenious construction of certain integral perfectoid algebras due to Andr\'{e} \cite[Section 2.5]{AndreDirectsummandconjecture}, see also \cite[Secton 2]{BhattDirectsummandandDerivedvariant}. Below we recall Andr\'{e}'s construction.

\subsection{Andr\'{e}'s technique of adjoining $p$-power roots}\label{subsec.Andre's construction}
Let $R$ be an integral perfectoid $K^\circ$-algebra and $g_1,\dots,g_n$ be a finite set of elements of $R$. We want to construct an almost faithfully flat extension $R\to R_{g_1,\dots,g_n}$ of integral perfectoid $K^\circ$-algebras such that each $g_i$ admits a compatible system of $p$-power roots in $R_{g_1,\dots,g_n}$.

We consider the integral perfectoid algebra $R_n=R\langle T_{g_1}^{1/p^\infty},\dots,T_{g_n}^{1/p^\infty}\rangle$ where the $T_{g_i}$ are indeterminates, and we let $Y=\Spa (R_n[\frac{1}{p}], R_n)$ be the associated perfectoid space. We set $R_{g_1,\dots,g_n}$ to be the integral perfectoid ring of functions on the Zariski closed subset of $Y$ defined by the ideal $(T_{g_1}-g_1,\dots,T_{g_n}-g_n)$. Explicitly, we have
\[
R_{g_1,\dots,g_n}=\widehat{\varinjlim}_l S_l, \text{\hspace{1em} where \hspace{1em}} S_l:=\O_Y^+(U(\frac{T_{g_1}-g_1,\dots,T_{g_n}-g_n}{p^l})).
\]
Here $U(\frac{T_{g_1}-g_1,\dots,T_{g_n}-g_n}{p^l})$ denotes the rational subset $\big\{y\in Y \;\big|\; |T_{g_i}(y)-g_i(y)|\leq |p^l|, \forall i\big\}$ and the completion is $p$-adic. Each $S_l$ is integral perfectoid. Using the explicit description of $R_{g_1,\dots,g_n}$ (which does not depend on the order of $g_1,\dots,g_n$) we have a map $R_{g_1,\dots,g_n}\to R_{g_1,\dots,g_n, h_1,\dots,h_m}$. Therefore $\{R_{g_1,\dots,g_n}\}$, where $g_1,\dots,g_n$ runs through all the finite sets of elements of $A$, naturally form a directed system of integral perfectoid $K^\circ$-algebras.

Note that in $R_{g_1,\dots,g_n}$, we have $T_{g_i}=g_i$ for each $i$ since $T_{g_i}-g_i$ is divisible by $p^l$ for all $l$ in $R_{g_1,\dots,g_n}$. Thus, each $g_i$ has a compatible system of $p$-power roots $g_i^{1/p^k}=T_{g_i}^{1/p^k}$ in $R_{g_1,\dots,g_n}$.  In the case that we only have one element $g=g_1$, the next lemma is \cite[Theorem 2.3]{BhattDirectsummandandDerivedvariant} (which was originally proved by Andr\'{e} in \cite[Section 2.5]{AndreDirectsummandconjecture} under a slightly different setup). A similar argument works for any finite sets of elements $g_1,\dots,g_n$ and we give details below.
\begin{lemma}[Andr\'{e}]
\label{lem.Andre}
For each $l>0$, the map $R/p^h \to S_l/p^h$ is almost faithfully flat for all $h>0$. Consequently, $R/p^h\to R_{g_1,\dots,g_n}/p^h$ is almost faithfully flat.
\end{lemma}
\begin{proof}
We first claim that it is enough to show $R/p^{1/p}\to S_l/p^{1/p}$ is almost faithfully flat. This follows from the more general fact:
\begin{claim}
\label{clm.faithfullyflatnzd}
Let $R\to S$ be a map of $K^\circ$-algebras (resp. $K^{\flat\circ}$-algebras). Suppose $R$, $S$ are both $t$-torsion free for some $0\neq t\in K^\circ$ (resp. $K^{\flat\circ}$). If $S/t$ is almost faithfully flat over $R/t$, then $S/t^h$ is almost faithfully flat over $R/t^h$ for all $h\geq 1$.
\end{claim}
\begin{proof}
We first prove $\Tor_1^{R/t^h}(S/t^h, M)$ is almost zero for all $R/t^h$-modules $M$. By considering $0\to tM\to M\to M/t\to 0$, it is enough to prove this when $M$ is annihilated by $t^{h-1}$. Since $t$ is a nonzerodivisor on both $R$ and $S$, we have $$M\otimes_{R/t^h}^\mathbf{L}S/t^h \cong M\otimes_{R/t^{h-1}}^\mathbf{L} (R/t^{h-1}\otimes_{R/t^h}^\mathbf{L}S/t^h)\cong M\otimes_{R/t^{h-1}}^\mathbf{L}S/t^{h-1}.$$ Therefore we are done by induction on $h$.

We next note that if $N$ is an $R/t^h$-module and $S/t^h\otimes_{R/t^h}N$ is almost zero, then $S/t\otimes_{R/t}{N/t}$ is almost zero. Hence $N/t$ is almost zero since $S/t$ is almost faithfully flat over $R/t$. But this implies $N=N/t^h$ is almost zero because $N/t^h$ has a finite filtration with each factor a quotient of $N/t$. This finishes the proof.
\end{proof}
By Scholze's approximation lemma \cite[Corollary 6.7]{ScholzePerfectoidspaces}, there exists $f_1,\dots,f_n$ in $R_n^\flat=R^\flat\langle (T_{g_1}^\flat)^{1/p^\infty},\dots,(T_{g_n}^\flat)^{1/p^\infty}\rangle$ such that
\begin{itemize}
\item $f_i^\sharp\equiv T_{g_i}-g_i$ mod $p^{1/p}$ for every $i$.
\item $U(\frac{T_{g_1}-g_1,\dots,T_{g_n}-g_n}{p^l})=U(\frac{f_1^\sharp,\dots,f_n^\sharp}{p^l})$ as rational subsets of $Y$.
\end{itemize}
Therefore $S_l=\O_Y^+(U(\frac{f_1^\sharp,\dots,f_n^\sharp}{p^l}))$ whose tilt is $\O_{Y^\flat}^+(U(\frac{f_1,\dots,f_n}{(p^\flat)^l}))$ by \cite[Theorem 6.3 (ii)]{ScholzePerfectoidspaces}. In particular, we have
$$S_l/p^{1/p}=\O_{Y^\flat}^+(U(\frac{f_1,\dots,f_n}{(p^\flat)^l}))/(p^\flat)^{1/p}.$$
Thus proving $S_l/p^{1/p}$ is almost faithfully flat over $R/p^{1/p}$is the same as proving $\O_{Y^\flat}^+(U(\frac{f_1,\dots,f_n}{(p^\flat)^l}))/(p^\flat)^l$ is $(p^\flat)^{1/p^\infty}$-almost faithfully flat over $R^\flat/(p^\flat)^{1/p}$.
Now by \cite[Lemma 6.4, third equation in the proof of (i)]{ScholzePerfectoidspaces}, we know that $\O_{Y^\flat}^+(U(\frac{f_1,\dots,f_n}{(p^\flat)^l}))$ is $(p^\flat)^{1/p^\infty}$-almost isomorphic to the $p^\flat$-adic completion of $$B:=R_n^\flat[u_1^{1/p^\infty},\dots,u_n^{1/p^\infty}]/(u_i^{1/p^k}(p^\flat)^{l/p^k} -f_i^{1/p^k}, \forall i, k).$$
Hence it is enough to show that $R^\flat/(p^\flat)^{1/p}\to B/(p^\flat)^{1/p}$ is almost faithfully flat.

At this point, we note that $B$ is the perfection of $$C:=R_n^\flat[u_1,\dots,u_n]/(u_1(p^\flat)^l-f_1,\dots, u_n(p^\flat)^l-f_n),$$ and by our choice of $f_i$, $$f_i={T^{\flat}_{g_i}}-{g^\flat_{i}} \text{ in }R_n^\flat/(p^\flat)^{1/p}= R_n/p^{1/p}$$
for some $g^\flat_{i}\in R_n^\flat$ (which amounts to choose a compatible sequence $\{g_{i,k}\}$ such that $g_{i,k}^{p^k}=g_{i}$ in $R_n/p^{1/p}$). Thus we have
$$C/(p^\flat)^{1/p}\cong R^\flat/(p^\flat)^{1/p}[(T_{g_1}^\flat)^{1/p^\infty},\dots,(T_{g_n}^\flat)^{1/p^\infty}][u_1,\dots,u_n]
/(T^{\flat}_{g_1}-g^\flat_{1},\dots,{T^{\flat}_{g_n}}-{g^\flat_{n}}).$$
It follows that $C/(p^\flat)^{1/p}$ is free over $R^\flat/(p^\flat)^{1/p}$ and that $(p^\flat)^{1/p}, u_1(p^\flat)^l-f_1,\dots, u_n(p^\flat)^l-f_n$ is a regular sequence on $R_n^\flat[u_1,\dots,u_n]$. In particular, $(p^\flat)^{1/p}$ is a nonzerodivisor on $C$.\footnote{We are using the fact that if $y_1,\dots, y_n$ is a regular sequence on a (possibly non-Noetherian, non-local) ring $R$, then $y_1$ is always a nonzerodivisor on $R/(y_2,\dots,y_n)$. By induction it comes down to the case $n=2$, where one can check directly \cite[Paragraph before Proposition 1.1.6]{BrunsHerzog}.} By \autoref{clm.faithfullyflatnzd}, $R^\flat\to C$ is almost faithfully flat mod any power of $(p^\flat)^{1/p}$. Since $R^\flat\to B=C_{\text{perf}}$ is the direct limit of the maps $R^\flat\xrightarrow{\text{Frob}^e} R^\flat\to C$ and the latter is almost faithfully flat mod any power of $(p^\flat)^{1/p}$ ($\text{Frob}^e$ is an isomorphism), $R^\flat\to B$ is almost faithfully flat mod any power of $(p^\flat)^{1/p}$ as desired.
\end{proof}

\subsection{Proof of \autoref{thm.ExistenceofAinfty} and its consequences}We now prove our main result in this section. We need the following construction from \cite{BhattDirectsummandandDerivedvariant}.
\begin{lemma}[Proposition 5.2 in \cite{BhattDirectsummandandDerivedvariant}]
\label{lem.BhattAinfty0}
Let $(A,\m)$ be a complete regular local ring of mixed characteristic $(0,p)$ and dimension $d$. Then there exists a map $A\to R$ such that $R$ admits the structure of an integral perfectoid $K^\circ$-algebra, and such that
\begin{enumerate}[(1)]
  \item $A\to R$ is almost flat.
  \item If $M\otimes_AR$ is almost zero, then $M=0$.
\end{enumerate}
\end{lemma}

If we compare \autoref{lem.BhattAinfty0} with \autoref{thm.ExistenceofAinfty}, the latter satisfies one extra condition (that all elements of $A$ have compatible system of $p$-power roots). Our strategy of the proof is to start with $R$ as in \autoref{lem.BhattAinfty0}, apply the construction in \autoref{subsec.Andre's construction} to $R$ for all the finite sets of elements $g_1,\dots,g_n$, and then take a (huge) completed direct limit. We will show that the conditions $(1)$ and $(2)$ in \autoref{lem.BhattAinfty0} are preserved under these constructions. Below we give details.

\begin{proof}[Proof of \autoref{thm.ExistenceofAinfty}]
We first construct $A\to R$, where $R$ is an integral perfectoid $K^\circ$-algebra as in \autoref{lem.BhattAinfty0}. For any finite set of elements $g_1,\dots,g_n$ of $A$, we have $R\to R_{g_1,\dots,g_n}$ as in \autoref{subsec.Andre's construction}. Since $\{R_{g_1,\dots,g_n}\}$ form a directed system, we set
$$A_\infty=\widehat{\varinjlim}R_{g_1,\dots,g_n},$$
where the direct limit is taken over all the finite sets of elements of $A$ and the completion is $p$-adic. Then $A_\infty$ is an integral perfectoid $K^\circ$-algebra such that all elements of $A$ have a compatible system of $p$-power roots in $A_\infty$. It remains to prove that $A_\infty$ satisfies properties $(b)$ and $(c)$ in \autoref{thm.ExistenceofAinfty}.

\subsection*{Proving (b)}Since $A_\infty/p^k=\varinjlim (R_{g_1,\dots,g_n}/p^k)$ and $R_{g_1,\dots,g_n}/p^k$ is almost flat over $R/p^k$ by \autoref{lem.Andre}, $R/p^k\to A_\infty/p^k$ is almost flat for all $k>0$. Since $A\to R$ is almost flat by \autoref{lem.BhattAinfty0}, it follows that $A/p^k\to A_\infty/p^k$ is almost flat for all $k>0$. This implies $A\to A_\infty$ is almost flat by \autoref{lem.AlmostFlatComplete}.

\subsection*{Proving (c)}
Suppose  that $M\otimes_AA_\infty$ is almost zero but $M\neq 0$. We can choose $A/I$ a nonzero cyclic $A$-submodule of $M$. Since $A_\infty$ is almost flat over $A$, $A/I\otimes_AA_\infty\to M\otimes_AA_\infty$ is an almost injection. Thus $A/I\otimes_AA_\infty$ is almost zero, and it follows that $A/\m\otimes_AA_\infty$ is almost zero since $A/I\otimes_AA_\infty\twoheadrightarrow A/\m\otimes_AA_\infty$. But $R/\m R\cong A/\m\otimes_AR$ is not almost zero by \autoref{lem.BhattAinfty0} (2), and since $R/p\to A_\infty/p$ is almost faithfully flat, it follows that $A/\m\otimes_AA_\infty\cong R/\m R\otimes_{R/p}A_\infty/p$ is not almost zero. This is a contradiction.
\end{proof}

\begin{corollary}
\label{cor.PropertiesofAinfty}
Let $(A,\m)$ be a complete regular local ring of mixed characteristic $(0,p)$ and dimension $d$. Let $A\to A_\infty$ be a map as in \autoref{thm.ExistenceofAinfty}. Then we have %such that $A_\infty$ admits the structure of an integral perfectoid $K^\circ$-algebra, and $A_\infty$ satisfies the conclusion of \autoref{thm.ExistenceofAinfty}. Then we have
\begin{enumerate}[(1)]
  \item If $x_1,\dots,x_d$ is a system of parameters of $A$, then it is an almost regular sequence on $A_\infty$, i.e., $(x_1, \ldots, x_i)A_\infty :_{A_{\infty}} x_{i+1} \over (x_1, \ldots, x_i)A_\infty$ is almost zero for every $0\leq i\leq d-1$ and $A_\infty \over (x_1, \ldots, x_{d})A_\infty$ is not almost zero.\footnote{Recall that $(x_1, \ldots, x_i)A_\infty :_{A_{\infty}} x_{i+1}$ denotes the set of elements of $A_{\infty}$ that multiply $x_{i+1}$ into $(x_1, \ldots, x_i)A_\infty$.}
  \item $A_\infty\otimes\Hom_A(N, M)$ is almost isomorphic to $\Hom_{A_\infty}(N\otimes A_\infty, M\otimes A_\infty)$ for all finitely generated $A$-modules $N$ and all $A$-modules $M$.
\end{enumerate}
\end{corollary}
\begin{proof}
$(1)$: $x_1,\dots,x_{d}$ is a regular sequence on $A$ because $A$ is regular. Since $A\to A_\infty$ is almost flat by \autoref{thm.ExistenceofAinfty} $(b)$, it follows that $(x_1, \ldots, x_i)A_\infty : x_{i+1} \over (x_1, \ldots, x_i)A_\infty$ is almost zero. Moreover, $A_\infty \over (x_1, \ldots, x_{d})A_\infty$ is not almost zero by \autoref{thm.ExistenceofAinfty} $(c)$.

$(2)$: Let $A^{\oplus l}\to A^{\oplus n}\to N\to 0$ be a presentation of $N$, then we have $A_\infty^{\oplus l}\to A_\infty^{\oplus n}\to N\otimes A_\infty\to 0$. We look at the following commutative diagram:
\footnotesize
\[\xymatrix{
0 \ar[r] & A_\infty\otimes\Hom_A(N, M) \ar[d]\ar[r] & A_\infty\otimes\Hom_A(A^{\oplus n}, M) \ar[d]^\cong \ar[r] & A_\infty\otimes\Hom_A(A^{\oplus l}, M) \ar[d]^\cong \\
0 \ar[r] & \Hom_{A_\infty}(N\otimes A_\infty, M\otimes A_\infty) \ar[r] & \Hom_{A_\infty}(A_\infty^{\oplus n}, M\otimes A_\infty) \ar[r] & \Hom_{A_\infty}(A_\infty^{\oplus l}, M\otimes A_\infty)
}
\]
\normalsize
The second line is exact, and the first line is almost exact because $A_\infty$ is almost flat over $A$ by \autoref{thm.ExistenceofAinfty} $(b)$. The conclusion follows by chasing the diagram. %$A_\infty\otimes\Hom_A(N, M)$ is almost isomorphic to $\Hom_{A_\infty}(N\otimes A_\infty, M\otimes A_\infty)$.
\end{proof}

\begin{remark}
\label{remark--recentprogressonBigCM}
We record here some very recent progress towards constructing integral perfectoid (almost) big Cohen-Macaulay algebras. For simplicity we will keep our notations: $A$ is a complete regular local ring of mixed characteristic $(0,p)$.
\begin{enumerate}
  \item In Bhatt's unpublished note \cite[Corollary 9.4.7]{BhattLectureNotesPerfectoidSpaces}, it is proved that there exists an integral perfectoid $A$-algebra $A_{\infty,\infty}$ that is almost flat over $A$ mod $p^k$, and such that $A_{\infty,\infty}$ is absolutely integrally closed: each monic polynomial has a root in $A_{\infty,\infty}$. In particular every element of $A_{\infty,\infty}$ admits a compatible system of $p$-power roots. It follows from the same proof of \autoref{thm.ExistenceofAinfty} that $A_{\infty,\infty}$ satisfies the conclusion of \autoref{thm.ExistenceofAinfty}. Bhatt's construction is similar to ours in spirit (i.e., iterate Andr\'{e}'s construction and take certain completed direct limit). %but is more involved (in order to obtain absolute integral closedness).
  \item Shimomoto \cite[Main Theorem 2]{ShimomotoIntegralperfectoidbigCMviaAndre} and Andr\'{e} \cite[Theorem 3.1.1]{AndreWeaklyFunctorialBigCM} proved that one can construct an integral perfectoid $A$-algebra $B$ that is a big Cohen-Macaulay $A^+$-algebra. Here $A^+$ denotes the absolute integral closure of $A$: the integral closure of $A$ in an algebraic closure of its fraction field. Since $B$ is an $A^+$ algebra, elements of $A$ have compatible system of $p$-power roots in $B$, and $B$ is (honestly) faithfully flat over $A$ because $A$ is regular and $B$ is big Cohen-Macaulay. Therefore $B$ also satisfies the conclusion of \autoref{thm.ExistenceofAinfty}. The construction of $B$ is quite difficult and relies on Andr\'{e}'s perfectoid Abhyankar lemma \cite{AndrePerfectoidAbhyankarLemma}.
\end{enumerate}
\end{remark}

\begin{setting}
\label{Setting}
Throughout the rest of this article. We fix our notations as follows:
\begin{itemize}
\item $(A,\m)$ will always denote a complete regular local ring of mixed characteristic $(0,p)$ and dimension $d$.
\item $A_\infty$ will always denote a fixed integral perfectoid $K^\circ$-algebra that satisfies the conclusion of \autoref{thm.ExistenceofAinfty}. The existence of such $A_\infty$ follows from \autoref{thm.ExistenceofAinfty} (see also \autoref{remark--recentprogressonBigCM}).
\item $p^{1/p^\infty}z=0$ means $z$ is \emph{almost zero}. More precisely, this means $p^{1/p^k}z=0$ for all positive integers $k$, or equivalently, $z$ is $(p^{1/p^\infty})$-torsion where $(p^{1/p^\infty})\subseteq K^\circ$ is the ideal that we use to measure almost mathematics.
\end{itemize}
We caution the reader that, although $A_\infty$ is reduced (since it is integral perfectoid), taking $p^e$-th roots in $A_\infty$ is generally not unique because we are working in mixed characteristic. In particular, elements of $A$ may have many compatible systems of $p$-power roots in $A_\infty$. This will be addressed carefully throughout this paper.
\end{setting}

We end this section with the following lemma.
\begin{lemma}
\label{lem.AlmostKernel=0}
Let $c\neq 0$ be an element of $A$, and fix any compatible system of $p$-power roots $\{c^{1/p^e}\}_{e=1}^\infty$ in $A_\infty$. Then
$$\{\eta\in H_\m^d(A) \;|\; c^{1/p^e}\eta=0 \in H_\m^d(A_\infty) \text{ for all $e$}\}=0.$$
\end{lemma}
\begin{proof}
Recall that $H_\m^d(A)$ is the injective hull of $A/\m$ (since $A$ is Gorenstein \cite[Proposition 3.5.4]{BrunsHerzog}), and by the \v{C}ech complex description of local cohomology \cite[page 130]{BrunsHerzog}, the socle element (up to multiplication by a unit) of $H_\m^d(A)$ can be expressed as $\frac{1}{x_1\cdots x_d}$ where $x_1,\dots,x_d$ is a regular system of parameters of $A$. Since the injective hull is an essential extension, we get that $\frac{1}{x_1\cdots x_d}$ lies in the left hand side if the latter is not zero.
Thus we have $\frac{c^{1/p^e}}{x_1\cdots x_d}=0$ in $H_\m^d(A_\infty)$ for every $e$. This means for every $e$, there exists a $w$ depending on $e$ such that
$$c^{1/p^e}(x_1\cdots x_d)^{w}\in (x_1^{w+1},\dots,x_d^{w+1})A_\infty.$$
Since $x_1,\dots,x_d$ is an almost regular sequence on $A_\infty$ by \autoref{cor.PropertiesofAinfty}, this implies
$p^{1/p^e}c^{1/p^e}\in (x_1,\dots,x_d)A_\infty$\footnote{In general, $z(x_1\cdots x_d)^{w}\in (x_1^{w+t},\dots,x_d^{w+t})A_\infty$ implies $p^{1/p^\infty}z\in (x_1^t,\dots,x_d^t)A_\infty$. The condition implies $x_1^w(z(x_2\cdots x_d)^w-ax_1^t)\in (x_2^{w+t},\dots,x_d^{w+t})A_\infty$ for some $a\in A_\infty$. So $p^{1/p^\infty}(z(x_2\cdots x_d)^w-a_1x_1^t) \in (x_2^{w+t},\dots,x_d^{w+t})A_\infty$ and thus $p^{1/p^\infty}z(x_2\cdots x_d)^w\in (x_1^t, x_2^{w+t},\dots,x_d^{w+t})A_\infty.$ Note that we have dropped the exponent of $x_1$ at the expense of multiplying by $p^{1/p^\infty}$. Do the same thing for $x_2,\dots,x_d$ consecutively, we have $p^{1/p^\infty}z\in (x_1^t,\dots,x_d^t)A_\infty.$ This fact will be used in Section 5.}for every $e$ and thus $pc\in (x_1,\dots,x_d)^{p^e}A_\infty$ for every $e$. Now we consider the following commutative diagram:
\[
\xymatrix{
\frac{A}{(x_1,\dots,x_d)^{p^e}} \ar[r] & \frac{A_\infty}{(x_1,\dots,x_d)^{p^e}A_\infty}\\
\frac{A}{(x_1,\dots,x_d)^{p^e}:pc } \ar[r] \ar@{^{(}->}[u]  & \frac{A_\infty}{((x_1,\dots,x_d)^{p^e}:pc)A_\infty} \ar[u]_\phi
}
\]
where the vertical maps send $1$ to $pc$. Since $pc\in (x_1,\dots,x_d)^{p^e}A_\infty$, we have $\phi=0$. %Notice that $\phi$ factors through $A_{\infty} \over ((x_1, \dots, x_d)^{p^e} :_{A_{\infty}} pc)$, and hence $\phi=0$.
On the other hand, $\phi$ is an almost injection because $A_\infty$ is almost flat over $A$ by \autoref{thm.ExistenceofAinfty}. Thus $\frac{A_\infty}{((x_1,\dots,x_d)^{p^e}:pc)A_\infty}$ is almost zero, so $\frac{A}{(x_1,\dots,x_d)^{p^e}:pc }=0$ by \autoref{thm.ExistenceofAinfty}. Therefore $pc\in (x_1,\dots,x_d)^{p^e}A$ for every $e$, which is a contradiction.
\end{proof}

\section{Perfectoid test ideals}
\label{sec.PerfMultiplierTestIdeals}

Our goal in this section is to introduce a mixed characteristic ideal which is analogous to the multiplier ideal which appeared in higher dimensional algebraic geometry and the test ideal which appeared independently in characteristic $p > 0$ commutative algebra \cite{LazarsfeldPositivity2,HochsterHunekeTC1}. We recall the reader that we are using the notations as in \autoref{Setting}. %That is, $A$ is a $d$-dimensional complete regular local ring of mixed characteristic $(0,p)$, and $A_\infty$ is an integral perfectoid algebra almost faithfully flat over $A$ and such that all elements of $A$ have a compatible system of $p$-power roots in $A_\infty$.

\begin{definition}
Fix an ideal $\fra \subseteq A$ and a real number $t \geq 0$.  We define
\[
\begin{array}{r}
0^{\mystar \fra^t}_{H_\m^d(A)}=\left\{\eta\in H_\m^d(A)\;|\; \forall e>0, p^{1/p^\infty} f^{1/p^e} \eta=0 \text{ in } H_\m^d(A_\infty) \text{ for all $f \in \fra^{\lceil tp^e \rceil}$ }\right\},
\end{array}
 \]
where $f^{1/p^e}$ denotes all possible $p^e$-th roots of $f$ in $A_\infty$ that are part of a compatible system of $p$-power roots.
If we also fix a sequence of elements $\{ f_1, \dots, f_n \}$ generating $\fra$ and for each $i$, one fixed compatible system of $p$-power roots $\{f_i^{1/p^e}\}_{e=1}^\infty$ for $f_i$ (these data we denote by $[f_1,\dots,f_n]$), then we set
\[
\begin{array}{rl}
0^{\mystar [f_1,\dots,f_n]^t}_{H_\m^d(A)}=\Big\{\eta\in H_\m^d(A)\;| & \forall e>0, p^{1/p^\infty} g\eta=0 \text{ in } H_\m^d(A_\infty) \\
& \text{for all $g=\prod_{i=1}^af_{j_i}^{1/p^e}$ where $a\geq tp^e$}\Big\},
\end{array}
 \]
In the above product, $f_{j_i}^{1/p^e}$ runs over our chosen elements in $\{f_1^{1/p^e}, \dots, f_n^{1/p^e} \}$ (note that $j_i$ belongs to the set $\{1,\dots,n\}$, and repeats are allowed).
We then define
\[
\mytau^\sharp(\fra^t)=\mytau^\sharp(A, \fra^t)=\Ann_A(0^{\mystar \fra^t}_{H_\m^d(A)})
\]
and
\[
\mytau^\sharp([f_1,\dots,f_n]^t)=\mytau^\sharp(A, [f_1,\dots,f_n]^t)=\Ann_A(0^{\mystar [f_1,\dots,f_n]^t}_{H_\m^d(A)}).
\]
\end{definition}

We will usually write $\mytau^\sharp([\underline{f}]^t)$ for $\mytau^\sharp([f_1,\dots,f_n]^t)$ when $f_1,\dots,f_n$ is clear from the context. In the case that $\fra = (f)$ is principal, we will write $\mytau^\sharp(f^t)$ for $\mytau^\sharp(\fra^t)$. We emphasize that a priori this is different from $\mytau^\sharp([f]^t)$. This is because in the former one considers {\it all} possible $p^e$-th roots of elements in $(f)$ in $A_\infty$ that are part of a compatible system of $p$-power roots, while in the latter one we {\it fix} a \emph{single} compatible system of $p$-power roots of $f$.

\begin{remark}
A key part of the definition of $0^{\mystar [f_1,\dots,f_n]^t}_{H_\m^d(A)}$ and $\mytau^\sharp([f_1,\dots,f_n]^t)$ is that we are not considering $p^e$-th roots of elements like $f_1 + f_2$, and this is important, because unlike in characteristic $p>0$, taking $p^e$-th roots is not additive in mixed characteristic. Another key point is that in the notation $[f_1, \dots, f_n]^t$ we control directly \emph{which} $p^e$-th roots we are taking.
\end{remark}

The following properties are straightforward from the definition.

\begin{proposition}
\label{prop.EasyContainments}
Fix $\{ f_1, \ldots, f_n\}$ a sequence of generators of an ideal $\fra \subseteq A$ and for each $f_i$ fix a compatible system of $p$-power roots of $f_i$ in $A_\infty$ (in order to define $\mytau^\sharp([f_1, \ldots, f_n]^t)$). Then we have
\begin{enumerate}
\item For every $t\geq 0$, $\mytau^\sharp([f_1, \ldots, f_n]^t) \subseteq \mytau^\sharp(\fra^t).$
\item If $t'>t$, then $\mytau^\sharp(\fra^{t'})\subseteq \mytau^\sharp(\fra^t)$ and $\mytau^\sharp([f_1, \ldots, f_n]^{t'}) \subseteq \mytau^\sharp([f_1, \ldots, f_n]^t)$.
\item If $\fra\subseteq \frb$, and $f_1,\dots,f_n,f_{n+1},\dots,f_m$ is a fixed set of generators of $\frb$ (we should also fix a compatible system of $p$-power roots of each $f_i$), then for every $t\geq 0$, $\mytau^\sharp(\fra^t)\subseteq \mytau^\sharp(\frb^t)$ and $\mytau^\sharp([f_1, \ldots, f_n]^t)\subseteq \mytau^\sharp([f_1, \ldots, f_m]^t)$.
\end{enumerate}
\end{proposition}

%\begin{proof}
%For the containment, simply observe that if $a \geq t p^e$ then any product $\prod_{i = 1}^a f_{j_i} \in \fra^{\lceil tp^e \rceil}$.
%\end{proof}

%\[
%\begin{array}{rl}
%& 0^{\mystar [f_1, \ldots, f_n]^t}_{H_\m^d(A)} \\
%= & \big\{\eta\in H_\m^d(A)\;|\; p^{1/p^\infty} f^{1/p^e} \eta=0 \text{ in } H_\m^d(A_\infty) \text{ for all $f = \prod_{i = 1}^a f_{j_i}$ and all $a/p^e \geq t$}\big\}
%\end{array}
% \]
%As before, we omit the subscript and write $0^{\mystar [f_1, \ldots, f_n]^t}$, and we define $\mytau([f_1, \ldots, f_n]^t)$ to be $\Ann_A(0^{\mystar [f_1, \ldots, f_n]^t}_{H_\m^d(A)})$.

%\begin{remark}
%We can make other variants of this definition.  For example, we could replace the $p^{1/p^{\infty}}$ with $g^{1/p^{\infty}}$ for some appropriate $g$.  Alternately we could replace $A_{\infty}$ with any number of other objects (indeed, as we saw in the introduction, test ideals and multiplier ideals are roughly defined by doing exactly that).
%\end{remark}

One of the key properties of multiplier and test ideals is the fact that small positive perturbations of the exponent do not change the ideal.  We do not know if this is true for $\mytau^\sharp$.

\begin{question}
Is it true that $\mytau^\sharp(\fra^t)=\mytau^\sharp(\fra^{t+\epsilon})$ or $\mytau^\sharp([f_1, \ldots, f_n]^t)=\mytau^\sharp([f_1, \ldots, f_n]^{t+\epsilon})$ for all $\epsilon \ll 1$?
\end{question}

Due to this, we make the following definition. This will be our real definition of perfectoid test ideals.

\begin{definition}
Fix $\{ f_1, \ldots, f_n\}$ a sequence of generators of an ideal $\fra \subseteq A$ and for each $f_i$ fix a compatible system of $p$-power roots of $f_i$ (in order to define $\mytau^\sharp([f_1, \ldots, f_n]^t)$). We define $\mytau(\fra^t)=\mytau(A, \fra^t)$ to be the union or sum of $\{\mytau^\sharp(\fra^{t'})\}$ for all $t'>t$. Since $\mytau^\sharp(\fra^{t'})\subseteq \mytau^\sharp(\fra^t)$ for all $t'>t$, by the Noetherian property of $A$, this is $\mytau^\sharp(\fra^{t+\epsilon})$ when $\epsilon\ll 1$. Similarly, we define $\mytau([f_1, \ldots, f_n]^t)=\mytau(A, [f_1, \ldots, f_n]^t)$ to be $\mytau^\sharp([f_1, \ldots, f_n]^{t+\epsilon})$ for $\epsilon\ll 1$.
\end{definition}

It follows from the definition and \autoref{prop.EasyContainments} that we have:
\begin{equation}
\label{equation--easy containment}
\tag{$\dagger$}
\mytau^\sharp(\fra^t)\supseteq \mytau^\sharp([f_1, \ldots, f_n]^t)\supseteq \mytau([f_1, \ldots, f_n]^t)\subseteq \mytau(\fra^t)\subseteq \mytau^\sharp(\fra^t)
\end{equation}

As with multiplier ideals and test ideals, in the notation $\mytau(\fra^t)$, $\fra^t$ is a formal object.  This can cause confusion when $t$ is an integer since, for instance, $\fra^2$ makes sense on its own.  Thus it is natural to ask whether $\mytau^\sharp((\fra^n)^t) = \mytau^\sharp(\fra^{nt})$ and whether $\mytau((\fra^n)^t) = \mytau(\fra^{nt})$. We \emph{do not} know how to do this with our definition of $\mytau^\sharp$, but as we shall see it is not difficult to show this for $\mytau(\fra^t)$ and $\mytau([f_1, \ldots, f_n]^t)$. This will be crucial for our later purposes (and suggests that $\mytau$ is a better definition than $\mytau^\sharp$).

\begin{lemma}
\label{lem:UseEBigForClosure}
In the definition of $0^{\mystar \fra^t}_{H_\m^d(A)}$ and $0^{\mystar [f_1,\dots,f_n]^t}_{H_\m^d(A)}$, one may restrict to $e \gg 0$.
\end{lemma}
\begin{proof}
Indeed, restricting the $e$ to those $\gg 0$ results in fewer conditions and hence a \emph{larger} subset of $H^d_{\fram}(A)$.  On the other hand suppose that $\eta$'s image in $H^d_{\fram}(A_{\infty})$ is annihilated by $p^{1/p^{\infty}}f^{1/p^e}$ for all $f \in \fra^{\lceil t p^e \rceil}$ where $f^{1/p^e}$ is part of a compatible system of $p$-power roots, for $e \geq e_0$.  Fix now some $n \geq 0$ and $f \in \fra^{\lceil tp^n \rceil}$.  Then for $e \geq e_0$,
\[
p^{1/p^{\infty}} f^{1/p^n} \eta = p^{1/p^{\infty}} (f^{p^e})^{1/p^{n+e}} \eta = 0
\]
since $f^{p^e} \in \fra^{p^e \lceil tp^n \rceil} \subseteq \fra^{\lceil t p^{e+n} \rceil}$.  This shows that $\eta \in 0^{\mystar \fra^t}_{H_\m^d(A)}$. A similar argument works for $0^{\mystar [f_1,\dots,f_n]^t}_{H_\m^d(A)}$ and we omit the details.
\end{proof}

\begin{proposition}
\label{proposition--test ideals of powers}
For all positive integers $n$, $\mytau((\fra^n)^{t})=\mytau(\fra^{nt})$.
\end{proposition}
\begin{proof}
Fix an $\varepsilon > 0$ so that $\mytau((\fra^n)^{t}) = \mytau^\sharp((\fra^n)^{t + \varepsilon})$ and $\mytau(\fra^{nt}) = \mytau^\sharp(\fra^{nt + \varepsilon})$. By definition, any $\mytau^\sharp((\fra^n)^{t + \varepsilon'})$ where $\varepsilon'\leq\varepsilon$ also computes $\mytau((\fra^n)^{t})$.
We first show that $0^{\mystar (\fra^n)^{t+\varepsilon}}_{H_\m^d(A)} \supseteq 0^{\mystar\fra^{nt+\varepsilon}}_{H_\m^d(A)}$. Note that $0^{\mystar\fra^{nt+\varepsilon}}_{H_\m^d(A)}$ consists of $\eta \in H^d_{\fram}(A)$ whose images in $H_\m^d(A_\infty)$ are annihilated by $p^{1/p^\infty}f^{1/p^e}$ with $f\in\fra^{\lceil (nt + \varepsilon) p^e \rceil}$ and $f^{1/p^e}$ part of a compatible system of $p$-power roots of $f$, while $0^{\mystar (\fra^n)^{t+\varepsilon}}_{H_\m^d(A)}$ consists of $\eta$ that are annihilated by $p^{1/p^\infty}g^{1/p^e}$, where $g \in (\fra^n)^{\lceil (t+\varepsilon) p^e \rceil}$ and $g^{1/p^e}$ part of a compatible system of $p$-power roots of $g$. Since $\lceil (nt+\varepsilon)p^e\rceil \leq n\lceil (t+\varepsilon)p^e \rceil$, one sees that $(\fra^n)^{\lceil (t + \varepsilon) p^e \rceil} \subseteq \fra^{\lceil (nt + \varepsilon) p^e \rceil}$ and so
$0^{\mystar (\fra^n)^{t+\varepsilon}}_{H_\m^d(A)}\supseteq 0^{\mystar\fra^{nt+\varepsilon}}_{H_\m^d(A)}$. Thus $\mytau((\fra^n)^{t + \varepsilon}) \subseteq \mytau(\fra^{nt + \varepsilon})$.

Conversely, note that
\[
{\lceil (nt + \varepsilon) p^e \rceil} \geq n\lceil (t + \varepsilon/(2 n))p^e \rceil
\]
for $e \gg 0$.  Thus
\[
\fra^{\lceil (nt + \varepsilon) p^e \rceil} \subseteq \fra^{n\lceil (t + \varepsilon/(2 n))p^e \rceil}
\]
and so arguing as above, and using \autoref{lem:UseEBigForClosure}, we see that
\[
0^{\mystar\fra^{nt+\varepsilon}}_{H_\m^d(A)} \supseteq 0^{\mystar (\fra^{n})^{t+\varepsilon/(2n)}}_{H_\m^d(A)} = 0^{\mystar (\fra^n)^{t+\varepsilon}}_{H_\m^d(A)}
\]
where the last equality follows from our choice of $\varepsilon$.  This finishes the proof.
\end{proof}

We next prove the analog result for $\mytau([f_1,\dots,f_n]^t)=\mytau([\underline{f}]^t)$.

%To see $0^{\mystar (\fra^n)^{t+\epsilon}}\subseteq 0^{\mystar\fra^{nt+\epsilon}}$, we note that if $a/p^e\geq nt+\epsilon$, then we can write $a=bn+c$ such that $b=\lfloor a/n\rfloor$ and $0\leq c\leq n-1$. Pick $e\gg0$ such that $c/p^e\leq \epsilon/2$, since $\frac{bn+c}{p^e}\geq nt+\epsilon$, we must have $bn/p^e\geq nt+\epsilon/2$ and thus $b/p^e\geq t+\epsilon/2n$. Therefore if $\eta$ is killed by all $p^{1/p^\infty}g^{1/p^e}$, where $g=\prod^bg_k$ with $g_k=\prod^nf_j$ and $b/p^e\geq t+\epsilon/2n$, then it is killed by $p^{1/p^\infty}f^{1/p^e}$ with $f=\prod^af_j$ and $a/p^e\geq nt+\epsilon$ in $H_\m^d(A_\infty)$. This proves that $0^{\mystar (\fra^n)^{t+\epsilon}}\subseteq 0^{\mystar\fra^{nt+\epsilon}}$.

\begin{proposition}
\label{proposition--test ideals of powers-GensVersion}
Fix $\{\underline{f}\}=\{ f_1, \ldots, f_n\}$ a sequence of generators of an ideal $\fra \subseteq A$ and for each $f_i$ fix a compatible system of $p$-power roots of $f_i$ to define $\mytau^\sharp([\underline{f}]^t)$. Set $\underline{f}^{\bullet n}$ to be the set of degree $n$ monomials in the $f_i$, and we use the product of the fixed compatible system of $p$-power roots of $f_i$ to build compatible system of $p$-power roots for $\underline{f}^{\bullet n}$. Then for all positive integers $n$ and real numbers $t \geq 0$, $\mytau([\underline{f}^{\bullet n}]^{t})=\mytau([\underline{f}]^{nt})$.
\end{proposition}
\begin{proof}
Fix an $\varepsilon > 0$ so that $\mytau([\underline{f}^{\bullet n}]^t) = \mytau^\sharp([\underline{f}^{\bullet n}]^{t + \varepsilon})$ and $\mytau([\underline{f}]^{nt}) = \mytau^\sharp([\underline{f}]^{nt + \varepsilon})$. By definition, any $\mytau^\sharp([\underline{f}^{\bullet n}]^{t + \varepsilon'})$ where $\varepsilon'\leq\varepsilon$ also computes $\mytau([\underline{f}^{\bullet n}]^{t})$. As in the proof of \autoref{proposition--test ideals of powers}, the containment $0^{\mystar [\underline{f}^{\bullet n}]^{t+\varepsilon}}_{H_\m^d(A)}\supseteq 0^{\mystar[\underline{f}]^{nt+\varepsilon}}_{H_\m^d(A)}$ follows from the fact that $\lceil (nt+\varepsilon)p^e\rceil \leq n\lceil (t+\varepsilon)p^e \rceil$.

Next we show $0^{\mystar [\underline{f}^{\bullet n}]^{t+\varepsilon}}_{H_\m^d(A)}\subseteq 0^{\mystar[\underline{f}]^{nt+\varepsilon}}_{H_\m^d(A)}$. If $a/p^e\geq nt+\varepsilon$, then we can write $a=bn+c$ such that $b=\lfloor a/n\rfloor$ and $0\leq c\leq n-1$. Pick $e\gg0$ such that $c/p^e\leq \varepsilon/2$, since $\frac{bn+c}{p^e}\geq nt+\varepsilon$, we must have $bn/p^e\geq nt+\varepsilon/2$ and thus $b/p^e\geq t+\varepsilon/2n$. Therefore if $\eta$ is annihilated by $p^{1/p^\infty}g$ in $H_\m^d(A_\infty)$ for all $g=\prod^bg_k^{1/p^e}$ with $g_k^{1/p^e}=\prod^nf_j^{1/p^e}$ and $b/p^e\geq t+\varepsilon/2n$, then it is annihilated by $p^{1/p^\infty}f$ with $f=\prod^af_j^{1/p^e}$ and $a/p^e\geq nt+\varepsilon$. This proves that
\[
0^{\mystar[\underline{f}]^{nt+\varepsilon}}_{H_\m^d(A)} \supseteq 0^{\mystar [\underline{f}^{\bullet n}]^{t+\varepsilon/(2n)}}_{H_\m^d(A)} = 0^{\mystar [\underline{f}^{\bullet n}]^{t+\varepsilon}}_{H_\m^d(A)}
\]
where again the last equality follows from our choice of $\varepsilon$. This finishes the proof.
\end{proof}

Our next goal is to show that $\fra\subseteq\mytau([f_1,\dots,f_n])$ for any set of generators $\{f_1,\dots,f_n\}$ of $\fra$ (and any fixed set of compatible system of $p$-power roots $\{f_i^{1/p^e}\}_{e=1}^\infty$ in order to define $\mytau([f_1,\dots,f_n])$). It would follow that $\fra\subseteq\mytau(\fra)$ by (\ref{equation--easy containment}). %To establish this result, we need the following lemma.

\begin{proposition}
\label{proposition--test ideal contains the original ideal}
Fix $\{\underline{f}\}=\{ f_1, \ldots, f_n\}$ a set of elements of $A$, and for each $f_i$ fix a compatible system of $p$-power roots of $f_i$ in $A_\infty$ to define $\mytau([\underline{f}])$. Then we have $(f_1, \dots, f_n) \subseteq \mytau([\underline{f}])=\mytau([\underline{f}]^1)$.  In particular, $\fra \subseteq \mytau(\fra)$ for any ideal $\fra \subseteq A$.
\end{proposition}
\begin{proof}
%We select $\varepsilon\ll 1$ such that $\mytau([\underline{f}])=\mytau^\sharp([\underline{f}]^{1+\varepsilon})$.
It is enough to show that $(f_1,\dots,f_n)$ annihilates $0^{\mystar [\underline{f}]^{1+\varepsilon}}_{H_\m^d(A)}$ for $0 < \varepsilon \ll 1$. Fix an $\eta \in 0^{\mystar [\underline{f}]^{1+\varepsilon}}_{H_\m^d(A)}$.  %, it is enough to show that $f_i\eta=0$ for each $i$.
%Since we can take $\varepsilon=1/p^e$ for $e\gg0$,
%We know that $p^{1/p^\infty}f_i^{1/p^e}f_i\eta=0$ in $H_\m^d(A_\infty)$ for all $e\gg0$.
%(note that here, $f^{1/p^e}$ is the element in the compatible system that we used to define $\mytau([\underline{f}])$).
%Therefore
Our hypothesis implies that $(p f_i)^{1/p^\infty}(f_i\eta)=0$ in $H_\m^d(A_\infty)$. Applying \autoref{lem.AlmostKernel=0} to $c=pf_i$, we have $f_i\eta=0$ in $H_\m^d(A)$, i.e., $f_i$ annihilates $\eta$.
\end{proof}

\begin{corollary}
\label{cor.TestIdealOfPrincipal}
For $0 \neq f \in A$, we have $\mytau([f])=\mytau^\sharp([f])= (f)$.
\end{corollary}
\begin{proof}
We know $(f)\subseteq \mytau([f])\subseteq\mytau^\sharp([f])$ by \autoref{proposition--test ideal contains the original ideal} and \autoref{equation--easy containment}, thus it suffices to show that $\mytau^\sharp([f])=\mytau^\sharp([f]^1) \subseteq (f)$.  But if $f\eta=0$, then $p^{1/p^\infty}f\eta=0$ in $H_\m^d(A_\infty)$ and so
%\footnote{Recall that $0 :_{H^d_{\fram}(A)} f$ is by definition the set of elements of $H^d_{\fram}(A)$ that are annihilated by $f$.  In other words, it is the $f$-torsion of $H^d_{\fram}(A)$.}
%$0 :_{H^d_{\fram}(A)} f\subseteq  0^{\mystar [f]^1}_{H_\m^d(A)}$,
$\{ \eta \in {H^d_{\fram}(A)}\;|\;  f \eta = 0 \} \subseteq  0^{\mystar [f]^1}_{H_\m^d(A)}$,
%$0 :_{H^d_{\fram}(A)} f\subseteq  0^{\mystar [f]^1}_{H_\m^d(A)}$,
therefore the result follows by applying Matlis duality.  %(note that $f=(f^{1/p^e})^{p^e}$ for every choice of $f^{1/p^e}$)
\end{proof}

\begin{corollary}
\label{cor.TauEpsilonIsA}
Fix $\{\underline{f}\}=\{ f_1, \ldots, f_n\}$ a sequence of generators of an ideal $\fra \subseteq A$, and for each $f_i$ fix a compatible system of $p$-power roots of $f_i$ to define $\mytau([\underline{f}])$. Then we have $\mytau([\underline{f}]^0)=\mytau(\fra^0)=A$.
\end{corollary}
\begin{proof}
Since $\mytau([\underline{f}]^0)\subseteq \mytau(\fra^0)\subseteq A$, it suffices to prove that $\mytau([\underline{f}]^0)=A$, that is, $0^{\mystar [\underline{f}]^\epsilon}_{H_\m^d(A)}=0$ when $\epsilon\ll1$. Suppose that $\eta\in 0^{\mystar [\underline{f}]^\epsilon}_{H_\m^d(A)}=0^{\mystar [\underline{f}]^{1/p^e}}_{H_\m^d(A)}$ for all $e\gg0$, then we have $p^{1/p^e}f_i^{1/p^e}\eta=0$ in $H_\m^d(A_\infty)$ for all $e\gg0$ and all $i$. Applying \autoref{lem.AlmostKernel=0} to $c=pf_i$, we have $\eta=0$.
\end{proof}

\section{The subadditivity theorem}

The goal in this section is to prove the subadditivity for $\mytau([\underline{f}]^t)$. This is in analogous to \cite{DemaillyEinLazSubadditivity} and \cite{HaraYoshidaGeneralizationOfTightClosure}. We do not know how to prove the subadditivity property for $\mytau(\fra^t)$. This is the main reason that we need to work with $\mytau([\underline{f}]^t)$ in our later applications. We start by introducing the mixed perfectoid test ideals.

\begin{definition}
\label{definition--mixed mixed}
Let $\{f_1,\dots,f_n\}$ and $\{g_1,\dots,g_m\}$ be fixed sets of elements of $A$. We also fix a compatible system of $p$-power roots $\{f_i^{1/p^e}\}_{e=1}^\infty$, $\{g_j^{1/p^e}\}_{e=1}^\infty$ for all $f_i$ and $g_j$ in $A_\infty$.  Let $t, s \geq 0$ be two real numbers.
\[
{\def\arraystretch{1.4}
\begin{array}{rccl}
0^{\mystar [\underline{f}]^t[\underline{g}]^s}_{H_\m^d(A)} &  = &
\big\{\eta\in H_\m^d(A) \hspace{0.5em} \;| & \hspace{0.5em} \forall e>0, p^{1/p^\infty} fg \cdot \eta=0 \text{ in } H_\m^d(A_\infty) \\
& & & \text{ for all $f = \prod_{i=1}^af_{j_i}^{1/p^e}$ and all $g =\prod_{i=1}^bg_{k_i}^{1/p^e}$,}\\\vspace{6pt}
& & & \text{ where $a\geq tp^e$ and $b\geq sp^e$} \big\},
\end{array}
}
 \]
We define $\mytau^\sharp([\underline{f}]^t[\underline{g}]^s)=\Ann_A0^{\mystar [\underline{f}]^t[\underline{g}]^s}_{H_\m^d(A)}$, and $\mytau([\underline{f}]^t[\underline{g}]^s)=\mytau^\sharp([\underline{f}]^{t+\epsilon}[\underline{g}]^{s+\epsilon})$ for $\epsilon\ll 1$.
\end{definition}

\begin{remark}
\label{remark--mixed mixed unambiguity}
For $\epsilon\ll1$, we have $\mytau([\underline{f}]^t[\underline{f}]^s)=\mytau^\sharp([\underline{f}]^{t+\epsilon}[\underline{f}]^{s+\epsilon})
=\mytau^\sharp([\underline{f}]^{t+s+\epsilon})=\mytau([\underline{f}]^{t+s})$ (compare with the proof of \autoref{proposition--test ideals of powers-GensVersion}).
\end{remark}

Before we prove our subadditivity theorem, we recall some notations which appear frequently in the study of Artinian modules in commutative algebra, as well as facts about $H^d_{\fram}(R)$.  These are well known to experts in commutative algebra, but we do not know of a good reference.

%First some notation, given modules

\begin{remark}[Annihilators of submodules of $H^d_{\fram}(A)$]
We will show, in our setting, that if $M \subseteq H^d_{\fram}(A)$ is a submodule and $J = \Ann_A M$, then
\begin{equation}
\label{eq.AnnihilatorOfLocalCMSubmodule}
M = \Ann_{H^d_{\fram}(A)}J := \{ \eta \in H^d_{\fram}(R)\;|\; J \eta = 0 \} = \text{the $J$-torsion of $H^d_{\fram}(A)$.}
\end{equation}
Note here (and in the proof below) we slightly abuse the notation of annihilators to select the $J$-torsion of a module.  We hope this will not cause substantial confusion.

Now we verify \eqref{eq.AnnihilatorOfLocalCMSubmodule}.  Recall that because $A$ is regular $H^d_{\fram}(A)$ is isomorphic to $E$, the injective hull of the residue field.  Next observe that $E$ is Artinian (as are all its submodules).  Because $A$ is complete and so isomorphic to the Matlis dual of $E$, the submodules $M$ of $E$ are Matlis dual to the quotients of $A$.  Now, the annihilator $J$ of $M$ is equal to the annihilator of its Matlis dual, and hence the Matlis dual of $M$ is $A/J$.  On the other hand, $M = \Hom_A(A/J, E)$ is the submodule of $E$ that $J$ annihilates.
\end{remark}

We are ready to prove our subadditivity theorem. Our proof is inspired from the proof of subadditivity for test ideals in characteristic $p>0$ given by S.~Takagi from \cite[Theorem 2.4]{TakagiFormulasForMultiplierIdeals}. The essential reason that the theorem holds is because $A_\infty$ is almost flat over $A$ by \autoref{thm.ExistenceofAinfty}.

\begin{theorem}[Subadditivity]
\label{theorem--subadditivity}
With notation as in \autoref{definition--mixed mixed}, we have $\mytau^\sharp([\underline{f}]^t[\underline{g}]^s)\subseteq\mytau^\sharp([\underline{f}]^t)\mytau^\sharp([\underline{g}]^s)$ so also $\mytau([\underline{f}]^t[\underline{g}]^s)\subseteq\mytau([\underline{f}]^t)\mytau([\underline{g}]^s)$. In particular, \begin{equation}
\label{equation--subadditivity}
\mytau([\underline{f}^{\bullet n}]^t)=\mytau([\underline{f}]^{tn})  \subseteq \mytau([\underline{f}]^t)^n
\end{equation}
for all $t\in\mathbb{R}_{\geq0}$ and all $n\in \mathbb{N}$, where we define $\mytau([\underline{f}^{\bullet n}]^t)$ as in \autoref{proposition--test ideals of powers-GensVersion}.
\end{theorem}
\begin{proof}
We first claim that it is enough to show that
\begin{equation}
\label{equation--inclusion inside the colon}
0^{\mystar [\underline{f}]^t[\underline{g}]^s}_{H_\m^d(A)}\supseteq
%0^{\mystar [\underline{g}]^s}_{H_\m^d(A)}:_{H_\m^d(A)} \mytau^\sharp([\underline{f}]^t) :=
\{ \eta \in H_\m^d(A) \;|\;  \mytau^\sharp([\underline{f}]^t) \eta \subseteq 0^{\mystar [\underline{g}]^s}_{H_\m^d(A)} \}.
\end{equation}
To see this claim, if $z\in\Ann_{H_\m^d(A)}\mytau^\sharp([\underline{f}]^t)\mytau^\sharp([\underline{g}]^s)$ then $\mytau^\sharp([\underline{f}]^t)z\subseteq \Ann_{H_\m^d(A)}\mytau^\sharp([\underline{g}]^s)=0^{\mystar [\underline{g}]^s}_{H_\m^d(A)}$ and thus %$z\in 0^{\mystar [\underline{g}]^s}_{H_\m^d(A)}:\mytau^\sharp([\underline{f}]^t)\subseteq 0^{\mystar [\underline{f}]^t[\underline{g}]^s}_{H_\m^d(A)}$.
$z\in \{ \eta \in H^d_{\fram}(A) \;|\; \mytau^\sharp([\underline{f}]^t) \eta \subseteq 0^{\mystar [\underline{g}]^s}_{H_\m^d(A)} \} \subseteq 0^{\mystar [\underline{f}]^t[\underline{g}]^s}_{H_\m^d(A)}$.
Therefore $$\Ann_{H_\m^d(A)}\mytau^\sharp([\underline{f}]^t)\mytau^\sharp([\underline{g}]^s)\subseteq 0^{\mystar [\underline{f}]^t[\underline{g}]^s}_{H_\m^d(A)}=\Ann_{H_\m^d(A)}\mytau^\sharp([\underline{f}]^t[\underline{g}]^s)$$ and thus $\mytau^\sharp([\underline{f}]^t[\underline{g}]^s)\subseteq\mytau^\sharp([\underline{f}]^t)\mytau^\sharp([\underline{f}]^s)$ as desired.

Next we prove \autoref{equation--inclusion inside the colon}. Suppose that $\eta\in {H_\m^d(A)}$ satisfies $\mytau^\sharp([\underline{f}]^t)\eta\ \subseteq 0^{\mystar [\underline{g}]^s}_{H_\m^d(A)}$. By definition we know that $p^{1/p^\infty}g\eta \cdot \mytau^\sharp([\underline{f}]^t)=0$ in $H_\m^d(A_\infty)$ for all $g =\prod_{i=1}^bg_{k_i}^{1/p^e}$ with $b\geq sp^e$. This means $p^{1/p^\infty}g \eta\in \Ann_{H_\m^d(A_\infty)}(\mytau^\sharp([\underline{f}]^t)A_\infty)$. Since $H_\m^d(A_\infty)\cong H_\m^d(A)\otimes A_\infty$ (which follows from the \v{C}ech complex description of local cohomology \cite[page 130]{BrunsHerzog} as $d=\dim A$), we know from \autoref{cor.PropertiesofAinfty} that $\Ann_{H_\m^d(A_\infty)}(\mytau^\sharp([\underline{f}]^t)A_\infty)=\Hom_{A_\infty}(A_\infty/\mytau^\sharp([\underline{f}]^t)A_\infty, H_\m^d(A_\infty))$ is almost isomorphic to
\[
A_\infty\otimes \Hom_A(A/\mytau^\sharp([\underline{f}]^t), H_\m^d(A))=A_\infty\otimes\Ann_{H_\m^d(A)}\mytau^\sharp([\underline{f}]^t)=A_\infty\otimes 0^{\mystar [\underline{f}]^t}_{H_\m^d(A)}.
\]
Therefore $p^{1/p^\infty}g \eta \in A_\infty\otimes 0^{\mystar[\underline{f}]^t}_{H_\m^d(A)}$, which means for every $k$ we can write
\[
p^{1/p^k}g\eta =a_1\eta_1+\cdots+a_l\eta_l
\]
where $\eta_i\in 0^{\mystar [\underline{f}]^t}_{H_\m^d(A)}$ and $a_i\in A_\infty$. So for all $f=\prod_{i=1}^af_{j_i}^{1/p^e}$ with $a\geq tp^e$,
\[
p^{1/p^{k'}}p^{1/p^k}fg \cdot\eta =a_1(p^{1/p^{k'}}f\eta_1)+\cdots+a_l(p^{1/p^{k'}}f\eta_l)=0.
\]
for all $k$, $k'$. Thus we know $p^{1/p^\infty}fg\cdot \eta=0$ for all $f=\prod_{i=1}^af_{j_i}^{1/p^e}$ and all $g =\prod_{i=1}^bg_{k_i}^{1/p^e}$ such that $a\geq tp^e$ and $b\geq sp^e$. Hence $\eta\in0^{\mystar [\underline{f}]^t[\underline{g}]^s}_{H_\m^d(A)} $ as desired.

Finally, \autoref{equation--subadditivity} follows from \autoref{proposition--test ideals of powers-GensVersion}, \autoref{remark--mixed mixed unambiguity} and the inclusion $\mytau^\sharp([\underline{f}]^t[\underline{g}]^s)\subseteq\mytau^\sharp([\underline{f}]^t)\mytau^\sharp([\underline{g}]^s)$ we just proved applied to $\underline{f}=\underline{g}$, and $t, s$ both equal to $t+\epsilon$ for $\epsilon\ll1$, plus an induction on $n$.
\end{proof}

%\begin{remark}
%Our proof above resembles the proof of subadditivity for test ideals given by S.~Takagi from \cite[Theorem 2.4]{TakagiFormulasForMultiplierIdeals} instead of the ones which rely upon the restriction theorem applied to the diagonal in $R \otimes_k R$ such as \cite{DemaillyEinLazSubadditivity} for multiplier ideals and \cite{HaraYoshidaGeneralizationOfTightClosure} for test ideals.  Indeed, we even obtained a restriction theorem above so one might hope it is possible to use a similar technique to prove subadditivity in mixed characteristic.  The problem is that unlike the equal characteristic results, we would need to take our tensor product $A \otimes A$ over a coefficient ring such as $W(k)$ (for example, perhaps over the p-adic integers $\bZ_p$).  In particular, the fact that certain elements, \ie $p$, could cross that tensor seemed to be an obstruction for us.
%\end{remark}

We could also define the mixed characteristic perfectoid test ideal for $\mytau^\sharp(\fra^t)$ in an analogous way:
\[
\begin{array}{rcl}
0^{\mystar \fra^t\frb^s}_{H_\m^d(A)} =
\big\{\eta\in H_\m^d(A) & |& p^{1/p^\infty} f^{1/p^e}g^{1/p^e} \eta=0 \text{ in } H_\m^d(A_\infty) \\
& & \text{ for all $f \in \fra^{\lceil t p^e \rceil}$ and all $g \in \frb^{\lceil sp^e \rceil}$} \big\},
\end{array}
 \]
where $f^{1/^e}$ and $g^{1/p^e}$ denote all possible parts of a compatible system of $p$-power roots of $f$ and $g$ respectively. We then define $\mytau^\sharp(\fra^t\frb^s)=\Ann_A0^{\mystar \fra^t\frb^s}_{H_\m^d(A)}$ and $\mytau(\fra^t\frb^s)=\mytau^\sharp(\fra^{t+\epsilon}\frb^{s+\epsilon})$ for $\epsilon\ll1$. In fact, working with this definition, one can also prove $\mytau^\sharp(\fra^s\frb^t)\subseteq\mytau^\sharp(\fra^s)\mytau^\sharp(\frb^t)$ following a very similar argument as in \autoref{theorem--subadditivity}. The problem is that, it is not clear to us whether $\mytau(\fra^t \fra^s) = \mytau(\fra^{t+s})$, and hence the second conclusion of the subadditivity theorem does not seem to work for $\mytau(\fra^t)$.

\section{Comparison with the blowup}

The goal of this section is to prove \autoref{lemma--asymptotic test ideal is contained in Q}, that is, $\mytau((I^{(hl)})^{1/l})\subseteq I$ for every radical ideal $I\subseteq A$ of height $h$ and every positive integer $l$. This follows from our core result \autoref{lem.BlowupContainmentGeneral}. \autoref{lemma--asymptotic test ideal is contained in Q} implies immediately that $\mytau([\underline{f}]^{1/l})\subseteq I$ for every fixed generating set $\{\underline{f}\}$ of $I^{(hl)}$ (and every fixed compatible system of $p$-power roots of $\underline{f}$) since we always have $\mytau([\underline{f}]^{1/l})\subseteq \mytau((I^{(hl)})^{1/l})$ by \autoref{equation--easy containment}.

Our key idea is to study how information about our perfectoid test ideal can be obtained by blowing up a finitely generated ideal $J$.  The situation is easier if $\sqrt{J}$ contains $p$ as then the blowup of $J A_{\infty}$ is admissible since it is trivial outside of $V(p)$.  This allows us to use Scholze's vanishing theorem for perfectoid spaces \cite[Proposition 6.14]{ScholzePerfectoidspaces} which tells us that passing to the blowup is essentially harmless, up to almost mathematics and factoring.  In the case that $J$ does not contain a power of $p$, however, we use a similar strategy to the one in \cite[section 6]{BhattDirectsummandandDerivedvariant}: we need to pass to certain enlargement of $A_{\infty}$ where a multiple of $p$ is contained in $J B$. Throughout this section, we will mostly work with $\mytau(\fra^t)$, but as discussed earlier, the final result \autoref{lemma--asymptotic test ideal is contained in Q} holds for $\mytau([\underline{f}]^t)$ as well simply because of \autoref{equation--easy containment}.

We start by proving a series of four crucial lemmas (\autoref{lemma--closure at a finite level}, \autoref{lemma--forcing elements in Ainfinity}, \autoref{lemma--closure defined passing to A<p^h/g>} and \autoref{lemma--closure is defined by associated perfectoid space}) that allow us to handle the case that the ideal we are blowing up does not contain a power of $p$.  The reader who is only interested in the case when the ideal we blow up contains a power of $p$ may wish to jump directly to \autoref{lem.BlowupContainmentGeneral} where the main result of the section is proven.  First we need a definition.

\begin{definition}
We define
\[
\begin{array}{r}
0^{\mystar \fra^t}_{[l,h]}=\big\{\eta\in H_\m^d(A)\;|\; p^{1/p^l} f^{1/p^h} \eta=0 \text{ in } H_\m^d(A_\infty) \text{ for all $f \in \fra^{\lceil tp^h \rceil}$}\big\},
\end{array}
\]
where as usual $f^{1/p^h}$ denotes all $p^h$-th roots of $f$ that are part of a compatible system of $p$-power roots of $f$ in $A_\infty$.
We use $0^{\mystar \fra^t}_{[l,\infty]}$ (resp. $0^{\mystar \fra^t}_{[\infty, h]}$) if we allow $h$ (resp. $l$) to range over all positive integers. Under this definition we have $0^{\mystar \fra^t}_{H_\m^d(A)}=0^{\mystar \fra^t}_{[\infty,\infty]}$.
\end{definition}

\begin{lemma}
\label{lemma--closure at a finite level}
We have $0^{\mystar \fra^{t}}_{H_\m^d(A)}=0^{\mystar \fra^{t}}_{[k,k]}$ for all $k\gg0$.
\end{lemma}
\begin{proof}
We have containments $\dots \subseteq 0^{\mystar \fra^t}_{[l+1, \infty]}\subseteq 0^{\mystar \fra^t}_{[l,\infty]}\subseteq\cdots$. Since $H_\m^d(A)$ is Artinian, $0^{\mystar \fra^t}_{H_\m^d(A)}=0^{\mystar \fra^t}_{[\infty,\infty]}=0^{\mystar \fra^t}_{[l,\infty]}$ for all $l\gg0$. Now we fix such an $l\gg0$, it follows from the proof of \autoref{lem:UseEBigForClosure} that we have containments $ \dots \subseteq 0^{\mystar \fra^t}_{[l, h+1]}\subseteq 0^{\mystar \fra^t}_{[l,h]}\subseteq\cdots.$ Thus by the Artinian property of $H_\m^d(A)$ again, $0^{\mystar \fra^t}_{[l,\infty]}=0^{\mystar \fra^t}_{[l,h]}$ for all $h\gg0$. Now take $k\geq\max\{l,h\}$. We have $0^{\mystar \fra^{t}}_{H_\m^d(A)}=0^{\mystar \fra^{t}}_{[k,k]}$.
\end{proof}

The next lemma is a slight generalization of \cite[Lemma 3.2]{HeitmannMaBigCohenMacaulayAlgebraVanishingofTor}. We recall that, for any element $g\in A_\infty$,  $A_\infty\langle\frac{p^n}{g}\rangle$ denotes the integral perfectoid algebra, which is the ring of bounded functions on the rational subset $\{x\in X \hspace{0.5em}| \hspace{0.5em} |p^n|\leq |g(x)|\}$ where $X=\Spa(A_\infty[1/p], A_\infty)$ is the perfectoid space associated to $(A_\infty[1/p], A_\infty)$.

%If $g\in A$, then for any fixed compatible system of $p$-power roots $\{g^{1/p^e}\}_{e=1}^\infty$ in $A_\infty$, $A_\infty\langle\frac{p^n}{g}\rangle$ can be is almost isomorphic to the $p$-adic completion of $A_\infty[(\frac{p^n}{g})^{\frac{1}{p^\infty}}]$ by \cite[Lemma 6.4]{ScholzePerfectoidspaces}.

\begin{lemma}
\label{lemma--forcing elements in Ainfinity}
Let $I=(p^c, y_1, \dots,y_s)$ be an ideal of $A$ (that contains a power of $p$). Let $g=p^mg_0\in A$ where $p\nmid g_0$, and consider the map $A\to A_\infty\to A_\infty\langle \frac{p^b}{g}\rangle$ for every positive integer $b$. Suppose the image of $z\in A_\infty$ is contained in $IA_\infty\langle \frac{p^b}{g}\rangle$. If $cp^a+m<b$, then for every $g^{1/p^a}$ that is part of a compatible system of $p$-power roots of $g$ in $A_{\infty}$, we have $p^{1/p^a} g^{1/p^a} z\in IA_\infty$.

In particular, if the image of $z$ is contained in $IA_\infty\langle \frac{p^b}{g}\rangle$ for all $b>0$, then $p^{1/p^a}g^{1/p^a}z\in IA_\infty$ for all $a>0$.
\end{lemma}
\begin{proof}
We fixed a compatible system of $p$-power roots of $g$, call it $\{g^{1/p^e}\}_{e=1}^\infty$ that contains the particular $g^{1/p^a}$. Since $A_\infty\langle\frac{p^n}{g}\rangle$ is almost isomorphic to the $p$-adic completion of $A_\infty[(\frac{p^n}{g})^{\frac{1}{p^\infty}}]$ by \cite[Lemma 6.4]{ScholzePerfectoidspaces}, we have $$p^{1/p^t}z\in I\widehat{A_\infty[(\frac{p^b}{g})^{1/p^{\infty}}]}$$ for some $t>a$.
The image of $p^{1/p^t}z$ inside  $A_\infty[(\frac{p^b}{g})^{1/p^{\infty}}]/p^{c}=\widehat{A_\infty[(\frac{p^b}{g})^{1/p^{\infty}}]}/p^{c}$ is contained in the ideal $(y_1,\dots, y_s)$. Therefore we can write
\begin{equation}
\label{equation--express z before clearing denominator}
p^{1/p^t}z=p^{c}f_0+y_1f_1+\cdots+y_sf_s
\end{equation}
where $f_0, f_1,\dots,f_s\in A_\infty[(\frac{p^b}{g})^{1/p^{\infty}}]$. Since this is a finite sum, there exist integers $k$ and $h$ such that $f_0, f_1,\dots,f_s$ are elements in $A_\infty[(\frac{p^b}{g})^{1/p^k}]$ of degree in $(\frac{p^b}{g})^{1/p^k}$ bounded by $p^kh$.

Next we claim that multiplying by $g_0^h$ in \autoref{equation--express z before clearing denominator} will clear all the denominators of the $f_i$. This is because every $g^{1/p^e}$ (that is part of the compatible system of $p$-power roots of $g$) has the form $p^{m/p^e}g_0^{1/p^e}$ for a certain $g_0^{1/p^e}\in A_\infty$ (that is part of a compatible system of $p$-power roots of $g_0$). To see this, simply observe that
\[
(g^{1/p^e})^{p^e}=g=p^mg_0, \text{ which implies } (\frac{g^{1/p^e}}{(p^{1/p^e})^{m}})^{p^e}=g_0\in A_\infty.
\]
Since $A_\infty$ is integrally closed in $A_\infty[1/p]$, we have $\frac{g^{1/p^e}}{(p^{1/p^e})^{m}}\in A_\infty$ whose $p^e$-th power is $g_0$.
One checks that after multiplying by $g_0^h$ to \autoref{equation--express z before clearing denominator} we get:
$$p^{1/p^t}g_0^hz\in(g_0^{h-(1/p^{a})},p^{(b-m)/p^a})\cdot (p^{c}, y_1,\dots,y_s)A_\infty.$$
From this we know:
$$p^{1/p^t}g_0^hz=g_0^{h-(1/p^a)}(p^{c}h_0+y_1h_1+\cdots+y_sh_s) \text{ in } A_\infty/p^{(b-m)/p^a},$$
where $h_0,h_1,\dots,h_s\in A_\infty$. Rewriting this we have
$$g_0^{h-(1/p^a)}(p^{1/p^t}g_0^{1/p^a}z-p^{c}h_0-y_1h_1-\cdots-y_sh_s)=0 \text{ in } A_\infty/p^{(b-m)/p^a}.$$
Since $p\nmid g_0$, $g_0$ is a nonzerodivisor on $A/p$. This implies $g_0^{h-(1/p^a)}$ is an almost nonzerodivisor on $A_\infty/p^{(b-m)/p^a}$ since $A\to A_\infty$ is almost flat by \autoref{thm.ExistenceofAinfty}. Hence $p^{1/p^t}g_0^{1/p^a}z-p^{c}h_0-y_1h_1-\cdots-y_sh_s$ is annihilated by $(p^{1/p^{\infty}})$ in $A_\infty/p^{(b-m)/p^a}$. In particular, since $t>a$, we know $$p^{1/p^{a}} g_0^{1/p^{a}}z\in (p^{c}, y_1,\dots, y_s) \text{ in } A_\infty/p^{(b-m)/p^a}.$$ Finally, note that $b>cp^a+m$ and thus $p^{(b-m)/p^a}$ is a multiple of $p^{c}$, and $g^{1/p^a}$ is a multiple of $g_0^{1/p^a}$. Therefore we have $$p^{1/p^{a}} g^{1/p^{a}}z\in(p^c,y_1,\dots,y_s)A_\infty.$$
This finishes the proof.
\end{proof}

The main technical statement which allows us to pass to the enlargement of $A_\infty$ is contained below.

\begin{lemma}
\label{lemma--closure defined passing to A<p^h/g>}
Let $p,x_1,\dots,x_{d-1}$ be a system of parameters of $A$. For all $\epsilon\ll1$ we have
\[
\begin{array}{rl}
0^{\mystar \fra^{t+\epsilon}}_{H_\m^d(A)} =\Big\{\frac{z}{p^cx_1^c\cdots x_{d-1}^c}\in H_\m^d(A)\;\Big|& \forall e>0, p^{1/p^{\infty}} f^{1/p^e} z\in (p^c, x_1^c, \dots, x_{d-1}^c)A_\infty\langle\frac{p^b}{g}\rangle \text{ for all}\\
& \text{$f \in \fra^{\lceil (t+\epsilon) p^e \rceil}$, all $0\neq g\in\fra$, and all integers $b>0$}\Big\},
\end{array}
\]
where $z\in A$, $c\in \mathbb{N}$, $f^{1/p^e} \in A_{\infty}$ runs over all possible $p^e$-th roots of $f$ that are part of a compatible system of $p$-power roots of $f$.
\end{lemma}

\begin{proof}
We first prove the containment $``\subseteq"$. This works for any $\epsilon>0$. Suppose $\frac{z}{p^cx_1^c\cdots x_{d-1}^c}\in 0^{\mystar \fra^{t+\epsilon}}_{H_\m^d(A)}$, then $\frac{p^{1/p^\infty}f^{1/p^e}z}{p^cx_1^c\cdots x_{d-1}^c}=0$ in $H_\m^d(A_\infty)$ for all $f\in \fra^{\lceil (t+\epsilon)p^e \rceil}$ and all $f^{1/p^e}$ that are part of a compatible system of $p$-power roots. This means for every $l>0$, $$p^{1/p^l}f^{1/p^e}z (px_1\cdots x_{d-1})^w\in (p^{c+w}, x_1^{c+w},\dots,x_{d-1}^{c+w})A_\infty$$ for some $w$ (which depends on $c$, $e$ and $l$). But since $p,x_1,\dots,x_{d-1}$ is an almost regular sequence on $A_\infty$ by \autoref{cor.PropertiesofAinfty}, this implies $$p^{1/p^{l-1}}f^{1/p^e}z \in (p^{c}, x_1^{c},\dots,x_{d-1}^{c})A_\infty$$ for all $f \in \fra^{\lceil (t+\epsilon) p^e \rceil}$ and all $l>0$. Hence its image in $A_\infty\langle\frac{p^b}{g}\rangle$ is contained in $(p^c, x_1^c, \dots, x_{d-1}^c)A_\infty\langle\frac{p^b}{g}\rangle$ for all $0\neq g\in\fra$ and all $b> 0$. Thus $p^{1/p^\infty}f^{1/p^e}z \in (p^c, x_1^c, \dots, x_{d-1}^c)A_\infty\langle\frac{p^b}{g}\rangle$.

Next we prove the other containment $``\supseteq"$. We take $\varepsilon_0\ll 1$ such that $0^{\mystar \fra^{t+\varepsilon_0}}_{H_\m^d(A)}$ computes $0^{\mystar \fra^{t+\epsilon}}_{H_\m^d(A)}$. We note that $\varepsilon_0$ depends only on $\fra$ and $t$, and for every $\varepsilon_1<\varepsilon_0$, $0^{\mystar \fra^{t+\varepsilon_1}}_{H_\m^d(A)}$ also computes $0^{\mystar \fra^{t+\epsilon}}_{H_\m^d(A)}$. We choose $k\gg0$ such that $0^{\mystar \fra^{t+\varepsilon_0}}_{[k,k]}=0^{\mystar \fra^{t+\varepsilon_0}}_{H_\m^d(A)}$ by \autoref{lemma--closure at a finite level}, and $\frac{\varepsilon_0}{2}p^k\geq t+\varepsilon_0$. We also observe that $k$ depends on $\fra$, $t$, $\varepsilon_0$ (and hence only depends on $\fra$ and $t$). We will show that $0^{\mystar \fra^{t+\varepsilon_0}}_{[k,k]}=0^{\mystar \fra^{t+\varepsilon_0}}_{H_\m^d(A)}=0^{\mystar \fra^{t+\varepsilon_0/2}}_{H_\m^d(A)}$ contains
\begin{equation}
\label{eq.MidProofP2kClosure}
\begin{array}{rl}
\big\{\frac{z}{p^cx_1^c\cdots x_{d-1}^c}\in H_\m^d(A)\;|& p^{1/p^{2k}} f^{1/p^{2k}} z\in (p^c, x_1^c, \dots, x_{d-1}^c)A_\infty\langle\frac{p^b}{g}\rangle \text{ for all}\\
 & \text{$f \in \fra^{\lceil (t+\varepsilon_0/2) p^{2k} \rceil}$, all $0\neq g\in\fra$, and all integers $b> 0$}\big\},
 \end{array}
\end{equation}
where $f^{1/p^{2k}}$ runs over all possible $p^{{2k}}$-th roots of $f$ that are part of a compatible system of $p$-power roots of $f$.
This will establish the $``\supseteq"$ because the object in \autoref{eq.MidProofP2kClosure} (when applied to $\epsilon=\varepsilon_0/2$) is larger than the object in the statement of \autoref{lemma--closure defined passing to A<p^h/g>} (since it requires fewer conditions).

So select an arbitrary $\frac{z}{p^cx_1^c\cdots x_{d-1}^c}$ in the set in \autoref{eq.MidProofP2kClosure}, we have
$$p^{1/p^{2k}} f^{1/p^{2k}} z\in (p^c, x_1^c, \dots, x_{d-1}^c)A_\infty\langle\frac{p^b}{g}\rangle$$
for all $b>0$. By \autoref{lemma--forcing elements in Ainfinity}, we get
\begin{equation}
\label{equation--forcing z to ideal of Ainfty}
(p^{1/p^{2k}} g^{1/p^{2k}}) p^{1/p^{2k}} f^{1/p^{2k}} z\in (p^c, x_1^c, \dots, x_{d-1}^c)A_\infty
\end{equation}
for all $f \in \fra^{\lceil (t+\varepsilon_0/2) p^{2k} \rceil}$, all $f^{1/p^{2k}}$ part of a compatible system of $p$-power roots of $f$, all $0\neq g\in\fra$, and all $g^{1/p^{2k}}$ part of a compatible system of $p$-power roots of $g$.

Finally, for every $\tilde{f}\in \fra^{\lceil(t+\varepsilon_0)p^k\rceil}$, and every $\tilde{f}^{1/p^k}$ part of a compatible system of $p$-power roots, we can write $\tilde{f}^{1/p^k}=\tilde{f}^{\frac{p^k-1}{p^{2k}}}\tilde{f}^{1/p^{2k}}$, where $\tilde{f}^{1/p^{2k}}$ is the $p^k$-th root of $\tilde{f}^{1/p^k}$ in the compatible system. We claim that $\tilde{f}^{p^k-1}\in \fra^{\lceil (t+\varepsilon_0/2) p^{2k} \rceil}$, this is because
\footnotesize
$$\lceil(t+\varepsilon_0)p^k\rceil(p^k-1)\geq \lceil(t+\varepsilon_0)p^{2k}-(t+\varepsilon_0)p^k\rceil=\lceil(t+\varepsilon_0/2)p^{2k}+(\frac{\varepsilon_0}{2}p^k-(t+\varepsilon_0))p^k\rceil\geq \lceil(t+\varepsilon_0/2)p^{2k}\rceil$$
\normalsize
by our choice of $k$. Now apply (\ref{equation--forcing z to ideal of Ainfty}) to $g=\tilde{f}\in\fra$ (and we use $\tilde{f}^{1/p^{2k}}$ as part of the compatible system of $p$-power roots of $g$) and $f=\tilde{f}^{p^k-1}$, we find that $p^{1/p^k}\tilde{f}^{1/p^k}z\in (p^c, x_1^c, \dots, x_{d-1}^c)A_\infty$. Thus $\frac{z}{p^cx_1^c\cdots x_{d-1}^c}$ is annihilated by $p^{1/p^k}\tilde{f}^{1/p^k}$ in $H_\m^d(A_\infty)$ for every $\tilde{f}^{1/p^k}$ part of a compatible system of $p$-power roots of $\widetilde{f}$ with $\tilde{f}\in \fra^{\lceil(t+\varepsilon_0)p^k\rceil}$. Hence $$\frac{z}{p^cx_1^c\cdots x_{d-1}^c}\in 0^{\mystar \fra^{t+\varepsilon_0}}_{[k,k]}=0^{\mystar \fra^{t+\varepsilon_0}}_{H_\m^d(A)}=0^{\mystar \fra^{t+\varepsilon_0/2}}_{H_\m^d(A)}$$ as desired.
\end{proof}

\begin{lemma}
\label{lemma--closure is defined by associated perfectoid space}
With the notations as in \autoref{lemma--closure defined passing to A<p^h/g>}, we have
\begin{align*}
0^{\mystar \fra^{t+\epsilon}}_{H_\m^d(A)} =& \Big\{\frac{z}{p^cx_1^c\cdots x_{d-1}^c}\in H_\m^d(A)\;\Big|\; p^{1/p^{\infty}} f^{1/p^e} z=0 \text{ in } \\
& \myH^0\left(\frac{A}{(p^c, x_1^c, \dots, x_{d-1}^c)}\otimes^\mathbf{L}\mathbf{R}\Gamma(X_\infty^{b,g}, A_\infty\langle\frac{p^b}{g}\rangle)\right) \\
& \text{ for all $f \in \fra^{\lceil (t+\epsilon) p^e \rceil}$, all $0\neq g\in\fra$, and all $b> 0$}\Big\}.
\end{align*}
Here we set $X_\infty^{b,g}=\Spa(A_\infty\langle\frac{p^b}{g}\rangle[1/p], A_\infty\langle\frac{p^b}{g}\rangle)$ to be the perfectoid space associated to $(A_\infty\langle\frac{p^b}{g}\rangle[1/p], A_\infty\langle\frac{p^b}{g}\rangle)$.
\end{lemma}
\begin{proof}
This is true by \autoref{lemma--closure defined passing to A<p^h/g>} and utilizing the fact that $A_\infty\langle\frac{p^b}{g}\rangle$ is almost isomorphic to $\mathbf{R}\Gamma(X_\infty^{b,g}, A_\infty\langle\frac{p^b}{g}\rangle)$ with respect to $(p^{1/p^\infty})$ by Scholze's vanishing theorem of perfectoid spaces \cite[Proposition 6.14]{ScholzePerfectoidspaces}.
\end{proof}

We need to recall one more fact well known to experts.

\begin{lemma}
\label{lem.AdmissibleBlowupFactoring}
Suppose that $B$ is an integral perfectoid algebra and that $J \subseteq B$ is a finitely generated ideal containing a power of $p$.  Set $X = \Spa(B[1/p], B)$ to be the perfectoid space associated to $(B[1/p], B)$.  Then the map of ringed spaces
\[
(X, \O_{X}^+\cong B) \to \Spec(B)
\]
factors through the blowup of $J$.
\end{lemma}
\begin{proof}
%See \cite[Chapter 8]{BhattLectureNotesPerfectoidSpaces} and also \cite[Proof of Proposition 6.2]{BhattDirectsummandandDerivedvariant}.
%\todo{Update some references}.
This is described in footnote \#8 in \cite[Proof of Proposition 6.2]{BhattDirectsummandandDerivedvariant} and as pointed out there is implicit in the description of adic spaces found in \cite[14.8]{GabberRameroFoundationsAlmostRingTheory}.
\end{proof}

We are ready to prove our core lemma in this section.

\begin{lemma}
\label{lem.BlowupContainmentGeneral}
Let $\pi : Y \to X = \Spec A$ be the blowup of some ideal $J \subseteq A$ such that $Y$ is normal and that $\fra \subseteq \sqrt{J}$.  Suppose that $E$ on $Y$ is a Weil divisor with $\pi(E)\subseteq V(J)$. Fix $t\in\mathbb{R}_{\geq0}$ and suppose that for every $e>0$ and every $f \in \fra^{\lceil tp^e \rceil}$,
\[
\Div_Y(f) \geq p^e E.
\]
Then $\mytau(\fra^t) \subseteq \Gamma(Y, \O_Y(K_{Y/X}-E)) \subseteq A$.
\end{lemma}

The strategy of the proof, at least in the case that $J$ contains a power of $p$, is to show we can factor the map $A \to \myR \Gamma(Y, \O_Y) \xrightarrow{\cdot f^{1/p^e}} \myR \Gamma(X_{\infty}, \O_{X_{\infty}})$ through $\myR \Gamma(Y, \O_Y(E))$ where $X_{\infty} = \Spa(A_\infty[1/p], A_\infty)$.  After this, we use the fact that $\myR \Gamma(X_{\infty}, \O_{X_{\infty}})$ is almost quasi-isomorphic to $A_{\infty}$ \cite[Proposition 6.14]{ScholzePerfectoidspaces}.  Thus our overall strategy is similar to (and inspired by) the proof of the \emph{derived} direct summand conjecture, \cite{BhattDirectsummandandDerivedvariant}.  Roughly speaking, the idea is that our divisor condition forces $f^{1/p^e} \O_Y(E)$ to pullback to something contained in $\O_{X_{\infty}}$, at least up to some issues of integrality. Below we give details.

\begin{proof}
Let $J=(z_1,\dots,z_m)$, write $Y = \Proj A \oplus J T \oplus J^2 T^2 \oplus \cdots$ and let $U_1,\dots, U_m$ be an affine cover of $Y$ with $U_j=Y \setminus V(z_jT) \cong \Spec A[\frac{z_1}{z_j}, \frac{z_2}{z_j},\dots,\frac{z_m}{z_j}]$.  We fix an $e$ and an element $f^{1/p^e}\in A_\infty$ such that $f^{1/p^e}$ is part of a compatible system of $p$-power roots of $f$ with $f\in\fra^{\lceil (t+\epsilon)p^e\rceil}\subseteq \fra^{\lceil tp^e \rceil}$.

Let $h \in A_\infty[z_j^{-1}]$ be such that $h\in f^{1/p^e}\O_Y(E)(U_j) \subseteq f^{1/p^e} A[z_j^{-1}]$ for some $j$ (note we are using the $\pi(E) \subseteq V(J))$. %(this exists because $\pi(E)$ lies in $V(J)$).
Since $\Div_Y(f)\geq p^e E$, we have $h^{p^e}\in \O_Y(U_j)$. For any such fixed $h$, we claim the following:

\begin{claim}
\label{clm.IntegralElementInBlowup}
There exists $h'\in\overline{(z_1,\dots,z_m)A_\infty} = \overline{J A_{\infty}}$ (where $\overline{\bullet}$ denotes integral closure of an ideal) such that, if $Y'_{\infty} \to \Spec A_\infty$ is the blow up of $(z_1,\dots,z_m, h')$ in $\Spec A_\infty$ and $\rho' : Y'_{\infty} \to Y$ is the induced affine map (see \autoref{lem.IntegralElementPartialNormalization}), then we have that $$h \in \O_{Y'_\infty}(\rho'^{-1} U_j) \subseteq A_{\infty}[z_j^{-1}].$$
\end{claim}

%\todo{{\bf Karl:} Now I'm getting worried, is the image of $h$ in $\O_{Y'_\infty}(\rho'^{-1} U_j)$ canonical?  Or does it depend on choices of roots }
\begin{proof}[Proof of Claim]
%Because the image of $E$ lies in $V(J)$, we have $$h \in f^{1/p^e} \cdot \O_Y(E)(U_j) \subseteq f^{1/p^e} \cdot \O_Y(U_j)[z_j^{-1}] = f^{1/p^e} \cdot A[z_j^{-1}],$$ and
By construction, we can write $h = {f^{1/p^e} w \over z_j^{p^l}}$ for some integer $l$ and some $w \in A$.  Since $$\frac{fw^{p^e}}{z_j^{p^{e+l}}}=h^{p^e}\in \O_Y(U_j),$$ we know that there exists some $d\gg0$ such that $fw^{p^e}z_j^{p^d-p^{e+l}} \in (z_1, \ldots, z_m)^{p^d}$. Fixing a compatible system of $p$-power roots of $w$ and $z_j$ in $A_{\infty}$, we note that
\[
(f^{1/p^d}w^{1/p^{d-e}}z_j^{(p^d-p^{e+l})/p^d})^{p^d}=fw^{p^e}z_j^{p^d-p^{e+l}}\in (z_1,\dots,z_m)^{p^d}.
\]
Thus $f^{1/p^d}w^{1/p^{d-e}}z_j^{(p^d-p^{e+l})/p^d} \in \overline{(z_1,\dots,z_m)A_\infty}$. We set $h'=f^{1/p^d}w^{1/p^{d-e}}z_j^{(p^d-p^{e+l})/p^d}$, and let $Y'_\infty$ be the blow up. We have
$$h=\frac{f^{1/p^e} w}{z_j^{p^l}}=\frac{(f^{1/p^d}w^{1/p^{d-e}}z_j^{(p^d-p^{e+l})/p^d})^{p^{d-e}}}{z_j^{p^{d-e}}}\in \O_{Y'_\infty}(\rho'^{-1} U_j) \subseteq A_{\infty}[z_j^{-1}].$$
This finishes the proof of the Claim.
\end{proof}

Because the module $f^{1/p^e}\O_Y(E)(U_j)$ is finitely generated over $\O_Y(U_j)$ for every $j$, we collect the generators for all $1\leq j\leq m$ and we call them $h_1,\dots, h_k$ (there are implicit $j$s we are suppressing). For each $h_i$ we construct $h_i' \in \overline{J A_{\infty}}$ as in \autoref{clm.IntegralElementInBlowup}.
Let $Y_\infty$ be the blow up of $(z_1,\dots, z_m, h_1',\dots,h_k')$ of $\Spec A_\infty$. Since each $h_i'$ is in the integral closure of $(z_1,\dots, z_m)A_\infty$, the inverse image of the $\{U_j\}$ forms an affine cover of $Y_\infty$ by \autoref{lem.IntegralElementPartialNormalization}.  We then have a factorization $Y_{\infty} \xrightarrow{\rho} Y \to X$ with $\rho$ affine, and for each $j$ we have a natural map $\O_Y(U_j) \to \O_{Y_\infty}(\rho^{-1} U_j)$.

\begin{claim}
\label{clm.BlowupFactor}
With notation as above, the canonical map
\[
\O_Y\to \rho_* \O_{Y_\infty}\xrightarrow{\cdot f^{1/p^e}} \rho_* \O_{Y_\infty}
\]
factors through $\O_Y(E)$.
\end{claim}
\begin{proof}[Proof of claim]
By \autoref{clm.IntegralElementInBlowup} and construction, we know that the $\O_Y(U_j)$-generators of $f^{1/p^e}\O_Y(E)(U_j)$ are contained in $\O_{Y'_\infty}(\rho'^{-1} U_j)$.  Hence
\[
f^{1/p^e}\O_Y(E)(U_j)\subseteq \O_{Y_\infty}(\rho^{-1}U_j)
\]
for every $1\leq j\leq m$,
%In other words, the natural map
which proves the claim.
\end{proof}

\subsection*{The case when $J$ contains a power of $p$}  At this point we are essentially done in the case that $J$ contains a power of $p$.  We note that by \autoref{lem.AdmissibleBlowupFactoring} applied to $B=A_\infty$ and $X=X_\infty=\Spa(A_\infty[1/p], A_\infty)$, we have a factorization $(X_\infty, \cO_{X_\infty}^+=A_\infty)\to Y_\infty\to \Spec A_\infty$
because $Y_\infty$ is the blow up of $A_\infty$ at a finitely generated ideal $(z_1,\dots,z_m, h_1',\dots,h_k')$ that contains a power of $p$ (by our hypotheses). Therefore we have a commutative diagram (we abuse notation and use $\Gamma_{\fram}(Y, \bullet)$ to denote the functor $\Gamma_{\fram}(\Gamma(Y, \bullet))$):
\[
{\footnotesize
\xymatrix@C=12pt{
H^d_{\fram}(A) \ar[r]^{\cdot f^{1/p^e}} \ar[d] & H_\m^d(A_\infty) \ar[dr] \ar[drr]^\phi &  {} \\
\myH^d \myR \Gamma_{\fram}(Y, \O_Y) \ar[rr]^{\cdot f^{1/p^e}} \ar@{.>}[dr] & {} & \myH^d \myR \Gamma_{\fram}(Y_{\infty}, \O_{Y_{\infty}}) \ar[r] & \myH^d \myR \Gamma_{\fram}(X_{\infty}, \O_{X_{\infty}}^+) \\
& \myH^d \myR \Gamma_{\fram}(Y, \O_Y(E)) \ar@{.>}[ur] & {}
}
}
\]
Here the existence of the dotted arrows follows from \autoref{clm.BlowupFactor}. Since, by \cite[Proposition 6.14]{ScholzePerfectoidspaces}, $\myR \Gamma(X_{\infty}, \O_{X_{\infty}}^+)$ is almost isomorphic to $A_\infty$, the map $\phi$ is an almost isomorphism. Hence elements in $0^{\mystar\fra^{t+\epsilon}}_{H_\m^d(A)}$ are precisely those $\eta\in H_\m^d(A)$ whose image is almost zero in $\myH^d \myR \Gamma_{\fram}(X_{\infty}, \O_{X_{\infty}}^+)$, when we vary over all $e > 0$ and all $f \in \fra^{\lceil (t+\epsilon)p^e \rceil}$ (and all $f^{1/p^e}$ that is part of a compatible system of $p$-power roots of $f$) in the above diagram. But this is the case if $\eta$ has trivial image in $\myH^d \myR \Gamma_{\fram}(Y, \O_Y(E))$ by the commutative diagram. Therefore we have
\[
0^{\mystar\fra^{t+\epsilon}}_{H_\m^d(A)} \supseteq \ker \big( H^d_{\fram}(A) \to  \myH^d \myR \Gamma_{\fram}(Y, \O_Y(E)) \big).
\]
However by local and Grothendieck duality, see for instance \cite{HartshorneResidues}, the Matlis dual of $\myH^d\mathbf{R}\Gamma_\m(Y, \cO_Y(E))$ is $\myH^{-d}(\myR \Gamma(Y, \omega_{Y}(-E)[d])$ and so the Matlis dual of the map $H^d_{\fram}(A) \to \myH^d\mathbf{R}\Gamma_\m(Y, \cO_Y(E))$ is $A \cong \omega_A \leftarrow \Gamma(Y, \omega_{Y}(-E))$.  It follows that
\[
\mytau(\fra^t)\subseteq \Ann_A\left(\ker\left(H_\m^d(A)\to \myH^d\mathbf{R}\Gamma_\m(Y, \cO_Y(E))\right)\right) = \Gamma(Y,\cO_Y(K_{Y/X}-E)) \subseteq A.
\]
Here we take $K_X = 0$ and $K_Y = K_{Y/X}$ as described in and before \autoref{def.RelCanonical}.
 This proves the case when $J$ contains a power of $p$.

\subsection*{The case when $J$ may not contain a power of $p$}
We now handle the general case. This is the place that we need to use \autoref{lemma--closure is defined by associated perfectoid space} (which relies on the technical \autoref{lemma--forcing elements in Ainfinity} and \autoref{lemma--closure defined passing to A<p^h/g>}).

Let $Y_\infty^{b,g} \to \Spec A_\infty\langle\frac{p^b}{g}\rangle$ be the blowup of the ideal $(z_1,\dots, z_m, h_1',\dots,h_k') A_\infty\langle\frac{p^b}{g}\rangle$, where $0\neq g \in \fra$.  So we have a commutative diagram:
\[
\xymatrix{
Y_\infty^{b,g} \ar[r] \ar[d] & Y_\infty \ar[r] \ar[d] & Y \ar[d] \\
\Spec A_\infty\langle\frac{p^b}{g}\rangle \ar[r] & \Spec A_\infty \ar[r] & \Spec A.
}
\]
Now, for each $0\neq g \in \fra \subseteq \sqrt{J}$, for some $l > 0$ we have $g^l \in J$ and so $p^{lb}$ is contained inside $J \cdot A_\infty\langle\frac{p^b}{g}\rangle\subseteq (z_1,\dots, z_m, h_1',\dots,h_k')A_\infty\langle\frac{p^b}{g}\rangle$. Therefore by \autoref{lem.AdmissibleBlowupFactoring} applied to $B=A_\infty\langle\frac{p^b}{g}\rangle$, for every $b$ and every $0\neq g\in \fra$ we have a factorization
\[
(X_\infty^{b,g}, \mathcal{O}_{X_\infty^{b,g}}^+=A_\infty\langle\frac{p^b}{g}\rangle)\to Y_\infty^{b,g} \to \Spec A_\infty\langle\frac{p^b}{g}\rangle
\]
where we use $X_\infty^{b,g}=\Spa(A_\infty\langle\frac{p^b}{g}\rangle[1/p], A_\infty\langle\frac{p^b}{g}\rangle)$ to denote the perfectoid space associated to $(A_\infty\langle\frac{p^b}{g}\rangle[1/p], A_\infty\langle\frac{p^b}{g}\rangle)$.

The above discussion shows that for every $f^{1/p^e}$ part of a compatible system of $p$-power roots of $f$ with $f \in \fra^{\lceil (t+\epsilon)p^e \rceil}$, every $0\neq g \in \fra$ and every positive integer $b$, we have the following commutative diagram:
{
\footnotesize
\[\xymatrix@C=18pt{
A \ar[r] \ar[d] & A_\infty \ar[d]\ar[rd]^{\cdot f^{1/p^e}} & & &\\
\mathbf{R}\Gamma(Y, \cO_Y) \ar[r] \ar@{.>}[rd] & \mathbf{R}\Gamma(Y_\infty, \cO_{Y_\infty}) \ar[rd]^{\cdot f^{1/p^e}} & A_\infty \ar[d]\ar[r]  &  A_\infty\langle\frac{p^b}{g}\rangle\ar[d] \ar[rd] & \\
  & \mathbf{R}\Gamma(Y, \cO_Y(E)) \ar@{.>}[r] & \mathbf{R}\Gamma(Y_\infty, \cO_{Y_\infty}) \ar[r] & \mathbf{R}\Gamma(Y_\infty^{b,g}, \cO_{Y_\infty^{b,g}}) \ar[r] & \mathbf{R}\Gamma(X_\infty^{b,g}, A_\infty\langle\frac{p^b}{g}\rangle)
}\]%
}%
\noindent
Again, the key point is that we have the dotted arrows in the above diagram, because we proved that the map $\cO_Y\to \cO_{Y_\infty}\xrightarrow{\cdot f^{1/p^e}} \cO_{Y_\infty}$ factors through $\cO_Y(E)$ by \autoref{clm.BlowupFactor} (up to pushforward by affine morphisms that we omit from the notation).

By \autoref{lemma--closure is defined by associated perfectoid space}, $\frac{z}{p^cx_1^c\cdots x_{d-1}^c}\in 0^{\mystar \fra^{t+\epsilon}}_{H_\m^d(A)}$ if and only if the image of $z$ is annihilated by $(p^{1/p^\infty})$ under the natural map induced from the top left to the right bottom of the above diagram
\[
\begin{array}{rl}
 & \frac{A}{(p^c, x_1^c, \dots, x_{d-1}^c)}\\
 \to & \myH^0\left(\frac{A}{(p^c, x_1^c, \dots, x_{d-1}^c)} \otimes^{\mathbf{L}}\mathbf{R}\Gamma(Y, \cO_Y(E))\right)\\
 \to & \myH^0\left(\frac{A}{(p^c, x_1^c, \dots, x_{d-1}^c)} \otimes^{\mathbf{L}} \mathbf{R}\Gamma(X_\infty^{b,g}, A_\infty\langle\frac{p^b}{g}\rangle)\right)
\end{array}
\]
for every $f^{1/p^e}$, every $0\neq g \in \fra$ and every $b>0$. But this is clearly the case if $z$ has trivial image in $\myH^0\left(\frac{A}{(p^c, x_1^c, \dots, x_{d-1}^c)} \otimes^{\mathbf{L}}\mathbf{R}\Gamma(Y, \cO_Y(E))\right)$.
Thus we have
\[
\varinjlim_c\ker\left( \frac{A}{(p^c, x_1^c, \dots, x_{d-1}^c)}\to \myH^0\left(\frac{A}{(p^c, x_1^c, \dots, x_{d-1}^c)} \otimes^{\mathbf{L}}\mathbf{R}\Gamma(Y, \cO_Y(E))\right) \right)
\]
is contained in  $0^{\mystar \fra^{t+\epsilon}}_{H_\m^d(A)}$.  Now, the above is precisely
\[
\ker\big(H_\m^d(A)\to \myH^d\mathbf{R}\Gamma_\m(Y, \cO_Y(E))\big)
\]
because $\varinjlim_c\myH^0\big( \frac{A}{(p^c, x_1^c, \dots, x_{d-1}^c)} \otimes^{\mathbf{L}}\mathbf{R}\Gamma(Y, \cO_Y(E)) \big)\cong \myH^d\mathbf{R}\Gamma_\m(Y, \cO_Y(E))$. Note that here the transition maps are $\frac{A}{(p^c, x_1^c, \dots, x_{d-1}^c)}\xrightarrow{\cdot px_1\cdots x_{d-1}} \frac{A}{(p^{c+1}, x_1^{c+1}, \dots, x_{d-1}^{c+1})}$, which follows from the \v{C}ech complex characterization of $\mathbf{R}\Gamma_\m(\bullet)$, see \cite[page 130]{BrunsHerzog}. Again by Matlis, local and Grothendieck duality, $\Ann_A\left(\ker\left(H_\m^d(A)\to \myH^d\mathbf{R}\Gamma_\m(Y, \cO_Y(E))\right)\right) = \Gamma(Y,\cO_Y(K_{Y/X}-E)) \subseteq A.$
Therefore we have
\[
\mytau(\fra^t)=\mytau^\sharp(\fra^{t+\epsilon})=\Ann_A(0^{\mystar \fra^{t+\epsilon}}_{H_\m^d(A)})\subseteq \Gamma(Y, \cO_Y(K_{Y/X}-E))
\]
which proves the lemma.
\end{proof}

We come to our main result of the section.  We first remind our reader of the general definition of a symbolic power of a radical ideal.

\begin{definition}
If $Q \subseteq A$ is a prime ideal, then the $n$th symbolic power of $Q$, denoted $Q^{(n)}$ is defined to be $Q^n A_Q \cap A$.

Suppose that $I \subseteq A$ is a radical ideal.  Suppose $I = Q_1 \cap \dots \cap Q_t$ is a decomposition of $I$ into minimal primes of $I$.  In this case the $n$th symbolic power of $I$, denoted $I^{(n)}$, is defined to be the intersection:
\[
Q_1^{(n)} \cap \dots \cap Q_t^{(n)}.
\]
\end{definition}

\begin{theorem}
\label{lemma--asymptotic test ideal is contained in Q}
If $I\subseteq A$ is a radical ideal such that each prime component has height $\leq h$ then we have
\[
\mytau((I^{(lh)})^{1/l}) \subseteq I
\]
for every $l$.
\end{theorem}
\begin{proof}
Let $\pi : Y\to\Spec A$ be the normalization of the blowup of $I \subseteq A$.  In particular, $\pi$ is the blowup of some $J = \overline{I^n}$ by \autoref{lem.BlowupOfNormalizedIdealPower}.  Since $A$ is regular, we let $D = \sum_i D_i$ denote the union of components of the inverse image of $V(I)$ which dominate components of $V(I)$.  If we write $I = Q_1 \cap \dots \cap Q_t$ a primary decomposition, then over the localization $A_{Q_i}$, we are simply blowing up a power the maximal ideal $Q_iA_{Q_i}$ in a regular local ring.  It follows that there is exactly one $D_i$ lying over each $V(Q_i)$.

Next notice that $I^{(lh)}$ is contained in $\sqrt{J}$ since neither symbolic powers or integral closures change the vanishing locus.  Furthermore, elements of
\[
(I^{(lh)})^{\lceil p^e / l \rceil} \subseteq I^{(p^e h)}
\]
vanish to order at least $p^e h$ on each $D_i$ by construction, and so we may apply \autoref{lem.BlowupContainmentGeneral} with $E := hD$.
It is then enough to show that
\[
\Gamma(Y, \cO_Y(K_{Y/X}-E)) \subseteq I.
\]

Thus we must compute the exceptional divisor $K_{Y/X}$.  Since regular local rings are pseudo rational by \cite[Section 4]{LipmanTeissierPseudorationallocalringsandatheoremofBrianconSkoda}, we see that $H^d_{\fram}(A) \to H^d_{\fram}(Y, \cO_Y)$ injects, and so by local and Grothendieck duality, the Matlis dual map $\Gamma(Y, \cO_Y(K_{Y/X})) \to A$ surjects.  It follows that $K_{Y/X} \geq 0$.

We now write
\[
K_{Y/X} = \sum_i a_i D_i + \text{other effective terms}
\]
and compute the integers $a_i$.  Since $D_i$ is the only exceptional divisor dominating a component $V(Q_i) \subseteq V(I)$, this can be done after localizing at $Q_i$ and so the statement reduces to computing the relative canonical divisor of the blowup of a regular local ring of dimension $h_i \leq h$ at its maximal ideal.  At that point we see that $a_i = h_i - 1\leq h-1$ by \autoref{lem.RelativeCanonicalOfBlowup}.  It follows immediately that $\Gamma(Y, \cO_Y(K_{Y/X}-E))_{Q_i} \subseteq \Gamma(Y, \O_Y(-D))_{Q_i} = Q_i$ and so $\Gamma(Y, \cO_Y(K_{Y/X}-E)) \subseteq I$ as desired.
\end{proof}

%\begin{corollary}
%\label{cor.TestIdealOfPrincipal}
%For any $0 \neq f \in A$, we have that
%\[
%\mytau^{+\epsilon}(A, f^1) = \mytau(A, f^{1+\epsilon}) = \langle f \rangle.
%\]
%\end{corollary}
%\begin{proof}
%Choose $E = \Div(f)$ and apply \autoref{lem.BlowupContainmentGeneral} to see the containment $\subseteq$ (frankly though, this containment can be done with much less work).  The containment $\supseteq$ is simply \autoref{proposition--test ideal contains the original ideal}.
%\end{proof}

\section{Relation with multiplier ideals}

We have defined $\mytau(\fra^t)=\mytau(A, \fra^t) \subseteq A$ and have shown it satisfies at least some formal properties similar to those of the multiplier ideal \cite{LazarsfeldPositivity2} (in this section we will always write $\mytau(A, \fra^t)$ to clarify which ring we are working with).  On the other hand, $A[1/p]$ is a ring of equal characteristic $0$ and so there exists a log resolution of $(\Spec A[1/p], \fra \cdot A[1/p])$ and so we can compute its multiplier ideal.  Thus it is natural to compare $\mytau(A, \fra^t) \cdot A[1/p]$ with $\mJ(A[1/p], (\fra \cdot A[1/p])^t)$.

\begin{theorem}
\label{prop.ComparisonWithMultiplier}
For a pair $(A, \fra^t)$, we have
\[
\mytau(A, \fra^{t}) \cdot A[1/p] \subseteq \mJ(A[1/p], (\fra \cdot A[1/p])^t).
\]
\end{theorem}
\begin{proof}
First let $J \subseteq A[1/p]$ be an ideal whose blowup produces a log resolution of \mbox{$(A[1/p], (\fra \cdot A[1/p])^{t})$} \cite{TemkinDesingularizationOfQuasiExcellentCharZero}.  Because a log resolution principalizes $\fra$, the blowup of $\fra \cdot J$ also produces the same log resolution of $(A[1/p], (\fra \cdot A[1/p])^t)$.  Since we may choose this log resolution to be an isomorphism outside of $\fra \cdot A[1/p]$ (since $A[1/p]$ is regular), we may assume that $\fra\cdot A[1/p] \subseteq \sqrt{J}$.  Consider $J' = J \cap A$ and notice that $\fra \subseteq (\fra \cdot A[1/p]) \cap A \subseteq \sqrt{J} \cap A \subseteq \sqrt{J'}$.
%Now we claim that $\sqrt{J} \cap A \subseteq \sqrt{J'}$.  Indeed, if $x \in \sqrt{J} \cap A$, then $x^n \in J$ and also $x^n \in A$, so $x^n \in J'$ and thus $x \in \sqrt{J'}$.   Putting this together $\fra \subseteq \sqrt{J'}$.
Let $\pi : Y \to X = \Spec A$ be the normalized blowup of $\fra \cdot J'$ (the blowup of $\overline{(\fra \cdot J')^n}$ for some $n > 0$ by \autoref{lem.BlowupOfNormalizedIdealPower}).  Write $\fra \cdot \cO_Y = \cO_Y(-G)$ and let $E = \lfloor t G \rfloor$. Finally write $U = \Spec A[1/p] \subseteq \Spec A$ and $V = \pi^{-1}(U)$.
%\begin{claim}
Now observe that $\pi|_V : V \to U$ is the blowup of $J$ and hence the log resolution we started with.
%\end{claim}
%\begin{proof}
%Note that the blowup of $J$ already principalizes $\fra \cdot A[1/p]$ and so the blowup of $\fra \cdot J$ is the same as the blowup $J$.  That blowup on the other hand is already normal and the claim is proved.
%\end{proof}
It follows that the hypotheses of \autoref{lem.BlowupContainmentGeneral} are satisfied and so $\mytau(A, \fra^{t}) \subseteq \Gamma(Y, \cO_Y(K_{Y/X}-\lfloor t G \rfloor))$.  But now by definition, $\Gamma(V, \cO_Y(K_{Y/X}-\lfloor t G \rfloor)) = \mJ(A[1/p], (\fra \cdot A[1/p])^t)$ and the result follows.
\end{proof}

We expect that the other containment should hold as well, namely:

\begin{conjecture}
$\mytau(A, \fra^{t}) \cdot A[1/p] = \mJ(A[1/p], (\fra \cdot A[1/p])^t)$
\end{conjecture}
We do not know how to prove it unfortunately, and it is certainly related to the question of localizing $\mytau$ which we also do not know how to handle.

Alternatively, in mixed characteristic one can define the multiplier ideal $\mJ(A, \fra^t)$ valuatively.  This is equivalent to defining
\[
\mJ(A, \fra^t) = \bigcap_{Y \to X} \Gamma(Y, \O_Y(K_{Y/X} - \lfloor t G_Y \rfloor ))
\]
where $Y \to X := \Spec A$ runs over all proper birational maps with $Y$ normal and such that $\fra \cdot \O_Y = \O_Y(-G_Y)$.  Note it is also not clear whether this definition commutes with localization (since the intersection is infinite).  With this definition, we have the following.

\begin{theorem}
\label{thm.TauInMultiplier}
Suppose that $(A, \fra^t)$ is a pair, then
\[
\mytau(A, \fra^t) \subseteq \mJ(A, \fra^t).
\]
\end{theorem}
\begin{proof}
This will follow from \autoref{lem.BlowupContainmentGeneral} but we must argue that we only need consider projective birational $\pi : Y \to X$ that are blowups of some ideal $J$ such that $\fra \subseteq \sqrt{J}$.  Moving from proper to projective maps (and hence blowups) is simply Chow's Lemma.  Since $A$ is regular local and hence a UFD, one may always choose $J$ so that $V(J)$ is the locus over which $\pi$ is not an isomorphism (\cf \cite[Chapter II, Exercise 7.11(c)]{Hartshorne}).  Thus it suffices to show that we can restrict our attention to $\pi$ that are an isomorphism outside of $V(\fra)$.
%We first argue that we can find such a projective birational map which is an isomorphism outside of $V(\fra)$.

To do this, first notice that our multiplier ideal can also be written as
\begin{equation}
\label{eq.ValuativeDescriptionOfMultiplierIdeal}
\mJ(A, \fra^t) = \{ f \in A \;|\; \text{for all $\pi : Y \to X$ as above, } \Div_Y(f) \geq \lfloor t G_Y \rfloor - K_{Y/X} \}.
\end{equation}
Next consider an arbitrary projective birational map $\pi: Y \to X$ with $Y$ normal.  Fix some prime divisor $D$ on $Y$ with corresponding discrete valuation $v$ and set $W = \pi(E)$ to be the center of $v$.  If $W$ is \emph{not} contained in $V(\fra)$, then the coefficient of $\lfloor t G_Y \rfloor - K_{Y/X}$ along $E$ is negative (since $K_{Y/X}$ is effective, see the proof of \autoref{lemma--asymptotic test ideal is contained in Q}), and so $E$ imposes no condition on the elements $f$ which make up $\mJ(A, \fra^t)$.  Thus, we only need to consider $E$ whose centers are contained in $V(\fra)$.  Next observe that the $E$ coefficient of $\lfloor t G_Y \rfloor - K_{Y/X}$ only depends on the valuation $v$, it does \emph{not} depend on the particular choice of $Y$.  Therefore, it suffices to show that if $v$ is a divisorial valuation whose center is contained in $V(\fra)$, then there is a projective birational map $\pi : Y \to X$, with $Y$ normal and $\pi$ an isomorphism outside of $V(\fra)$ such that the valuation ring of $v$ is the local ring of some prime divisor on $Y$.

However, \cite[Section 5]{ArtinNeronModels} (see also \cite[Proposition 3]{AbhyankarValuationsCentered}), shows that by repeatedly blowing up the center of the valuation $v$ one eventually obtains a birational model which realizes the valuation ring of $v$ as the localization at a generic point of a prime divisor.  Compositions of such blowups are isomorphisms away from the center of $v$, which is contained in $V(\fra)$.
%However, \cite[Proposition 3]{AbhyankarValuationsCentered} shows this at least over the further localization $A_Q$ where $V(Q) = W$.  Thus choose an ideal $J_Q \subseteq A_Q$ corresponding to this blowup.  Since $\sqrt{J_Q} = Q$, we have that $\fra A_Q \subseteq \sqrt{J_Q}$.  Let $J' \subseteq A$ be some ideal so that $J' A_Q = J_Q$ and set $J = J' + \fra^m$ for $m \gg 0$.  The blowup of this ideal $J$ is an isomorphism outside of $V(\fra)$ and still localizes to $J_Q$.  Thus our desired valuation $v$ also appears as a prime divisor on the blowup of $J$.

%We first must argue that it is sufficient to only consider $\pi : Y \to X$ that are isomorphisms outside of $V(\fra)$.
The Theorem then follows by \autoref{lem.BlowupContainmentGeneral}.
\end{proof}

\section{Asymptotic perfectoid test ideals}

We let $\{\fra_n\}_{n=1}^\infty$ be a graded sequence of ideals, i.e., $\fra_n\fra_m\subseteq\fra_{m+n}$ for all $m,n$.  By analogy with for instance \cite{EinLazSmithSymbolic}, it would be natural to define the $n$-th asymptotic perfectoid test ideal of the graded sequence $\{\fra_n\}$ as the ideal
\[
\mytau_\infty(\fra_n)=\sum \mytau(\fra_{ln}^{\frac{1}{l}}).
\]
However, since we don't know the subadditivity theorem for $\mytau(\fra^t)$, this version of asymptotic test ideal will not give us the desired result on symbolic powers of ideals. Instead we have to build the definition using $\mytau([\underline{f}]^t)$. Now since everything depends on the choice of the generating set $\{\underline{f}\}$ and a choice of compatible system of $p$-power roots of $\{\underline{f}\}$, we must proceed carefully.

Let $\{\fra_n\}_{n=1}^\infty$ be a graded sequence of ideals of $A$. We will be mostly interested in the situation that $\fra_n=I^{(n)}$, the $n$-th symbolic power of a radical ideal $I\subseteq A$. We define a generating set of each $\fra_n$ inductively as follows:

First, we let $\{\underline{f}_{(1)}\}=\{f_{(1),1},\dots, f_{(1),n_1}\}$ be a fixed generating set of $\fra_1$, and we also fix a compatible system of $p$-power roots for each $f_{(1),i}$ in $A_\infty$ (that we will use to define $\mytau([\underline{f}_{(1)}]^t)$). Now, suppose that a generating set $\{\underline{f}_{(s)}\}$ of $\fra_s$ and a compatible system of $p$-power roots of $\underline{f}_{(s)}$ have been chosen for all $s<m$. We let $$\{\underline{f}_{(m)}\}=\{f_{(m),1},\dots, f_{(m),n_m}\}$$ be a generating set of $\fra_m$ satisfying the condition that it contains all possible $f_{(s),i}f_{(t),j}$, where $s+t=m$ and $f_{(s),i}$, $f_{(t),j}$ are part of the chosen generating set of $\fra_s$ and $\fra_t$. Moreover, we fix a compatible system of $p$-power roots for each $f_{(m),k}$, such that, if $f_{(m),k}=f_{(s),i}f_{(t),j}$, then we use the product of the compatible system of $p$-power roots of $f_{(s),i}$ and $f_{(t),j}$, i.e., we let $${f_{(m),k}}^{1/p^e}={f_{(s),i}}^{1/p^e}{f_{(t),j}}^{1/p^e}.$$ It might happen that $f_{(s),i}f_{(t),j}=f_{(m),k}=f_{(s'),i'}f_{(t'),j'}$ for $s+t=m=s'+t'$, but the product of the chosen compatible system of $p$-power roots of $f_{(s),i}$ and $f_{(t),j}$ is not the same as the product of the chosen compatible system of $p$-power roots of $f_{(s'),i'}$ and $f_{(t'),j'}$. In this case, we simply allow $f_{(m),k}$ to appear multiple times in the generating set, but we use different compatible system of $p$-power roots, one coming from $f_{(s),i}f_{(t),j}$ and the other coming from $f_{(s'),i'}f_{(t'),j'}$.

We have defined a generating set $\{\underline{f}_{(m)}\}$ for each $\fra_m$ in the graded sequence, as well as a compatible system of $p$-power roots for each element $f_{(m),i}$ appearing in the generating set. Now we give our definition of asymptotic perfectoid test ideal.

\begin{definition}[Asymptotic perfectoid test ideals]
\label{definition--asympototic mixed}
We define
\[
\mytau_\infty([\underline{f}_{(n)}])=\sum_{l=1}^\infty \mytau([\underline{f}_{(ln)}]^{1/l}).
\]
\end{definition}

By \autoref{proposition--test ideals of powers-GensVersion},
$$\mytau([\underline{f}_{(ln)}]^{\frac{1}{l}})=\mytau([\underline{f}_{(ln)}^{\bullet m}]^{\frac{1}{ml}})$$
where as in the notation of \autoref{proposition--test ideals of powers-GensVersion}, $\underline{f}_{(ln)}^{\bullet m}$ denotes the set of all degree $m$ monomials in $\{\underline{f}_{(ln)}\}=\{f_{(ln),1},\dots, f_{(ln),n_{ln}}\}$. More importantly, by our choice of the generating set $\{\underline{f}_{(mln)}\}$ (and the way we fix the compatible system of $p$-power roots of elements in the generating set), we have
\[
\mytau([\underline{f}_{(ln)}^{\bullet m}]^{\frac{1}{ml}})\subseteq \mytau([\underline{f}_{(mln)}]^{\frac{1}{ml}})
\]
by \autoref{prop.EasyContainments}. Therefore we have
\[
\mytau_\infty([\underline{f}_{(n)}])=\mytau([\underline{f}_{(ln)}]^{1/l}) \text{ for all sufficiently large and divisible $l$}.
\]

\begin{proposition}
\label{proposition--subadditivity for asymptotic}
We have $\mytau_\infty([\underline{f}_{(mn)}])\subseteq \mytau_\infty([\underline{f}_{(n)}])^m$ for all $n, m\in \mathbb{N}$.
\end{proposition}
\begin{proof}
By the above discussion, we can choose $l$ sufficiently large and divisible such that $\mytau_\infty([\underline{f}_{(mn)}])=\mytau([\underline{f}_{(lmn)}]^{\frac{1}{l}})$ and $\mytau_\infty([\underline{f}_{(n)}])=\mytau([\underline{f}_{(lmn)}]^{\frac{1}{ml}})$. Now we have
\[
\mytau_\infty([\underline{f}_{(mn)}])=\mytau([\underline{f}_{(lmn)}]^{\frac{1}{l}})=\mytau([\underline{f}_{(lmn)}^{\bullet m}]^{\frac{1}{ml}})
\subseteq\mytau([\underline{f}_{(lmn)}]^{\frac{1}{ml}})^m = \mytau_\infty([\underline{f}_{(n)}])^m
\]
where the second equality follows from \autoref{proposition--test ideals of powers-GensVersion}, and the only inclusion is by the subadditivity property \autoref{theorem--subadditivity}. This completes the proof.
\end{proof}

\begin{theorem}
\label{thm.SymbolicContainmentCompleteLocalMixed}
Let $I\subseteq A$ be a radical ideal such that each prime component has height $\leq h$. Then we have $I^{(hn)}\subseteq I^n$ for all $n\in \mathbb{N}$.
\end{theorem}
\begin{proof}
Let $\{\fra_n\}=\{I^{(n)}\}$ be the graded sequence of ideals. We select a generating set $\{\underline{f}_{(n)}\}$ for each $\fra_n$ in this graded sequence as well as a compatible system of $p$-power roots for each element $f_{(n),i}$ appearing in the generating set as in the discussion before \autoref{definition--asympototic mixed}, and we form the asymptotic perfectoid test ideal $\mytau_\infty([\underline{f}_{(n)}])$ as in \autoref{definition--asympototic mixed}.

Since $\{\underline{f}_{(hn)}\}$ is a generating set of $\fra_{hn}=I^{(hn)}$, by \autoref{proposition--test ideal contains the original ideal} we have $$I^{(hn)}=\fra_{hn}\subseteq \mytau([\underline{f}_{(hn)}])\subseteq\mytau_\infty([\underline{f}_{(hn)}])$$
where the last containment follows from \autoref{definition--asympototic mixed}. But by \autoref{proposition--subadditivity for asymptotic}, we know that \[
I^{(hn)}\subseteq\mytau_\infty([\underline{f}_{(hn)}])\subseteq \mytau_\infty([\underline{f}_{(h)}])^n
\]
for all $n$. Therefore we are done if we can show that $\mytau_\infty([\underline{f}_{(h)}])\subseteq I$. However, $$\mytau_\infty([\underline{f}_{(h)}])=\mytau([\underline{f}_{(lh)}]^{1/l})$$
for some sufficiently large and divisible $l$, and since $\{\underline{f}_{(lh)}\}$ is a generating set for $\fra_{lh}=I^{(lh)}$, we know from \autoref{equation--easy containment} and \autoref{lemma--asymptotic test ideal is contained in Q} that
$$\mytau([\underline{f}_{(lh)}]^{1/l})\subseteq \mytau((I^{(lh)})^{1/l})\subseteq I,$$
This finishes the proof.
\end{proof}

\begin{theorem}
\label{thm.SymbolicContainmentInGeneral}
Let $R$ be a Noetherian regular ring with reduced formal fibers ({\itshape e.g.} $R$ is excellent) and let $I \subseteq R$ be a radical ideal such that each minimal prime of $I$ has height $\leq h$.  Then for every integer $n > 0$,
\[
I^{(hn)} \subseteq I^n.
\]
\end{theorem}
\begin{proof}
Since the formation of symbolic powers commutes with localization, it is enough to prove $I^{(hn)} \subseteq I^n$ after localizing at each prime ideal of $R$ and so we may assume that $R$ is local. Since $I^{(hn)}\widehat{R}\subseteq (I\widehat{R})^{(hn)}$ and $I\widehat{R}$ is still a radical ideal, see \cite[7.6.7(ii)]{EGA_IV_II}, and each minimal prime has height $\leq h$, see \cite[7.1.10]{EGA_IV_II} or \cite[Corollary on page 251]{MatsumuraCommutativeRingTheory}. If we can show that $(I\widehat{R})^{(hn)}\subseteq (I\widehat{R})^n=I^n\widehat{R}$, then it follows that $I^{(hn)}\subseteq I^n\widehat{R}\cap R=I^n$. Hence we may assume that $A := \widehat{R}$ is a complete regular local ring. In the case that $A$ is of equal characteristic, the result is already known by \cite[Theorem 1.1 (a)]{HochsterHunekeComparisonOfSymbolic} (also see \cite[Theorem A]{EinLazSmithSymbolic}). If $A$ has mixed characteristic, then we are done by \autoref{thm.SymbolicContainmentCompleteLocalMixed}. This completes the proof.
\end{proof}

\section{An example}

We can compute this perfectoid test ideal in a simple case.

\begin{example}[SNC pair]
\label{ex.SNCPair}
Consider $A = W(k)\llbracket x_1, \ldots, x_{d-1}\rrbracket$ for $k$ some perfect field of characteristic $p > 0$ and let $f = p^{a_0} x_1^{a_1} x_{d-1}^{a_2} \cdots x_{d}^{a_{d-1}}$ for some integers $a_i$.  Suppose that $A_{\infty}$ contains a fixed copy of $A_0 := A[p^{1/p^\infty}, x_1^{1/p^\infty},\dots, x_{d-1}^{1/p^\infty}]$, this follows for instance if one constructs $A_{\infty}$ via the $R$ from \cite[Proposition 5.2]{BhattDirectsummandandDerivedvariant} as stated in
\autoref{lem.BhattAinfty0}.  In particular for our compatible system of $p$-power roots of $f$, we fix the ones given in $A_{0}$ via products of roots of monomials.

We claim
\[
\mytau(A, [f]^t) = ( p^{\lfloor a_0 t \rfloor} \cdot x_1^{\lfloor a_1 t \rfloor} \cdots x_{d-1}^{\lfloor a_{d-1} t \rfloor} ).
\]
Since in the definition of $\mytau([f]^t)$, we are building the $+\epsilon$ variant of the perfectoid test ideal, we may work with a fixed $t + \epsilon = b/p^e$.  Consider $A' = A[p^{1/p^e}, x_1^{1/p^e}, \dots, x_{d-1}^{1/p^e}]$ and observe it is also regular and contains $f^{b/p^e}$.

Since $A'\subseteq A_{0}$, we have a factorization
\[
A \to A' \to A_{\infty}.
\]
Also note that we can factor the map $A \xrightarrow{\cdot f^{b/p^e}} A_{\infty}$ as
\[
A \to A' \xrightarrow{\cdot f^{b/p^e}} A' \to A_{\infty}
\]
Since $A'$ is regular, we have by \autoref{lem.AlmostKernel=0} that $$0 = \{ \eta \in H^d_{\fram}(A') \;|\; 0 = p^{1/p^{\infty}} \eta \in H^d_{\fram}(A_{\infty}) \}.$$
By using this and the argument of \autoref{cor.TestIdealOfPrincipal}, we have that
\[
%0 :_{H^d_{\fram}(A')} f^{b/p^e}
\{ \eta \in H^d_{\fram}(A') \;|\; f^{b/p^e} \eta = 0 \in H^d_{\fram}(A_{\infty}) \}
= \{ \eta \in H^d_{\fram}(A') \;|\; p^{1/p^{\infty}} f^{b/p^e} \eta = 0 \in H^d_{\fram}(A_{\infty}) \}.
\]
It then follows from local duality that
\[
\mytau(A, [f]^{t})=\mytau^\sharp(A, [f]^{t + \epsilon}) = \Phi(f^{t+\epsilon} A')
\]
where $\Phi$ is the generator of $\Hom_A(A', A)$ as an $A'$-module (for example, it can be taken to be the map which sends the monomial basis $p^{a_0/p^e} x_1^{a_1/p^e} \cdots x_{d-1}^{a_{d-1}/p^e} \mapsto p^{(a_0 - p^e + 1)/p^e} x_1^{(a_1 - p^e + 1)/p^e} \cdots x_{d-1}^{(a_{d-1}-p^e + 1)/p^e}$ if that term makes sense in $A$ and zero otherwise).  But this image is precisely $( p^{\lfloor a_0 t \rfloor} \cdot x_1^{\lfloor a_1 t \rfloor} \cdots x_{d-1}^{\lfloor a_{d-1} t \rfloor} )$ as desired (at this point, it is the same computation as the one for the test ideal).
\end{example}

\section{Further questions}
\label{sec.FurtherQuestions}

We record some open questions regarding the results herein.
\begin{question}
Fix $\{ f_1, \ldots, f_n\}$ a sequence of generators of an ideal $\fra \subseteq A$ and for each $f_i$ fix a compatible system of $p$-power roots of $f_i$  in $A_\infty$ (in order to define $\mytau^\sharp([f_1, \ldots, f_n]^t)$). Is any inclusion in \autoref{equation--easy containment}:
$$\mytau^\sharp(\fra^t)\supseteq \mytau^\sharp([f_1, \ldots, f_n]^t)\supseteq \mytau([f_1, \ldots, f_n]^t)\subseteq \mytau(\fra^t)\subseteq \mytau^\sharp(\fra^t)$$
an equality?
\end{question}

%Among all the definitions of $\mytau^\sharp$ and $\mytau$, the easiest for us to work with (that leads to the application in Section 7) seems to be $\mytau([f_1, \ldots, f_n]^t)$. But $\mytau([f_1, \ldots, f_n]^t)$ a priori depends on the collection $\{f_1,\dots,f_n\}$ (and even more, it depends on the fixed choice of a compatible system of $p$-power roots of $f_i$ in $A_\infty$). Therefore it would be good to give a definition that is independent of the generators yet will still establish the results we obtained in Section 7.

Another fundamental question left open in this paper is:

\begin{question}
Is $\mytau([\underline{f}]^t)$ or $\mytau(\fra^t)$ independent of the choice of $A_{\infty}$?  %Are there additional assumptions one could make on $A_{\infty}$-like objects that give us the same $\mytau([\underline{f}]^t)$ or $\mytau(\fra^t)$ or perhaps better multiplier/test ideal like objects?
\end{question}

%It is an important open question in positive characteristic whether or not the test ideal $\tau(R)$ formed, in an appropriate generalization of the analogy with our introduction, with respect to $R^+$ is the same as as the test ideal formed with respect to some variant of almost mathematics in $R^{1/p^{\infty}}$, where ``almost'' is taken with respect to $c^{1/p^{\infty}}$, a tower of $p$-power roots of a test element.  These are known to coincide when $R$ is Gorenstein (or various generalizations of Gorenstein).

%In characteristic zero, the analogous statement is that the multiplier ideal is independent of the resolution.  A key difference is that in mixed or positive characteristic, we don't know independence in as large a case as we might desire, but we do have an analog of a resolution which works for all choices of ideals.

We also ask how our object behaves with respect to localization.  Note that it is still an open question whether or not the formation of the classical characteristic $p > 0$ test ideal commutes with localization.

\begin{question}
If $Q \in \Spec A$ is a prime containing $p$ and $\fra$, is it true that
$$\mytau(A, \fra^t) \cdot \widehat{A_Q} = \mytau(\widehat{A_Q}, (\fra \widehat{A_Q})^t) ?$$
$$\mytau(A, [\underline{f}]^t)\cdot \widehat{A_Q} = \mytau(\widehat{A_Q}, [\underline{f}]^t)? $$
\end{question}

%It would also be quite reasonable to try to compute some examples.

%\begin{question}
%Can we compute $\mytau(\fra^t)$ or $\mytau([\underline{f}]^t)$ for some other examples?  For instance if $\fra = (f)$ where $f = y^2 - x^3$ or $p^2 - x^3$?  It is not hard to obtain bounds for such examples using blowup, via our \autoref{lem.BlowupContainmentGeneral}.
%\end{question}

%We also define the following analogously with the characteristic zero and $p$ definitions.

%\begin{definition}
%The \emph{Perfectoid pure threshold} of a pair $(A, \fra)$ (relative to a choice of $A_{\infty}$) is the supremum over all $t\geq 0$ such that $A = \mytau(\fra^t)$.
%\end{definition}

%It follows from \autoref{cor.TauEpsilonIsA} that the perfectoid pure threshold is well defined and is always greater than zero.  One then asks:

%\begin{question}
%Is the perfectoid pure threshold always a rational number?
%\end{question}

%This is known but not trivial for an analogous notion in equal characteristic $p > 0$ \cite{BlickleMustataSmithDiscretenessAndRationalityOfFThresholds}. It follows from \autoref{ex.SNCPair} that the perfectoid pure threshold of the pair $(A, [p^{a_0}x_1^{a_1} \cdots x_{d-1}^{a_{d-1}}])$ equals $1/\max\{ a_i \}$ just like for the log canonical or $F$-pure threshold.

%It would be natural to write down a Skoda-type theorem or a restriction theorem in this setting.  We believe we can do this for $\mytau([\underline{f}]^t)$.  It would also be natural to define such ideals for singular rings $R$.  We will address these issues in a future work.

\appendix
\section{Blowups}

In this appendix we briefly recall (and in some cases prove) facts about blowups of ideals.  These are well known but we record them here for ease of the reader.  Note, we are working with potentially non-Noetherian rings in most cases.

\begin{setting}
\label{set.Blowup}
Throughout this section, $R$ will be a reduced ring and $J \subseteq R$ will be a \emph{finitely generated} ideal.  We let $X = \Spec R$ and let $Y \to X$ be the blowup of $J$ in $X$.  In particular, set $S = R \oplus JT \oplus JT^2 \oplus \dots$ where the $T$ serve as a dummy variable to help distinguish degree, and thus $Y = \Proj S$.
\end{setting}

\begin{lemma}
If $J = ( z_1, \ldots, z_m )$, then the complements $U_i$ of $V(z_i T) \subseteq Y$ form an affine cover of $Y$ with $U = \Spec R[z_1/z_i, \ldots, z_m/z_i]$.
\end{lemma}
In the above $R[z_1/z_i, \ldots, z_m/z_i]$ is viewed as the subring of elements of $S[ (z_i T)^{-1} ]$ of the form $g T^n / (z_i T)^n$ as in \cite[Tag 052P]{stacks-project}.
\begin{proof}
Note any homogeneous prime of $S$ does not contain some $z_i$ and so this follows from for instance \cite[Tag 0804]{stacks-project}.
\end{proof}

%\begin{lemma}
%\label{lem.BlowupPowerAgreement}
%The blowup $Y \to X$ of $J$ agrees agrees with the blowup $Y_n \to X$ of $J^n$.
%\end{lemma}
%\begin{proof}
%In the case of a Noetherian scheme, this is simply \cite[Chapter II, Exercise 7.11(a)]{Hartshorne}.  More generally, we simply %note that if $J = \langle z_1, \ldots, z_m \rangle$ then $J^n = \langle z_1^n, z_1^{n-1} z_2, \dots, z_m^n\rangle$.  By the definition of a prime ideal, one may choose $U_{n,i}$ the complement of the $V(z_i^n T ) \subseteq Y_n$ to be an open affine cover of $Y_n$.  But
%\[
%\begin{array}{rl}
%& U_{n,i}  \\
%\cong & \Spec R[z_1^n/z_i^n, \dots, z_1 {z_i}^{n-1}/z_i^n, \dots, z_j z_i^{n-i} / z_i^n, \dots, z_m z_i^{n-i} / z_i^n, \dots, z_m^n/z_i^n]\\
% = & \Spec R[z_1/z_i, \ldots, z_m/z_i] \\
% \cong & U_i
%\end{array}
%\]
%and it is easy to see that this identification glues.
%\end{proof}

\begin{lemma}
\label{lem.IntegralElementPartialNormalization}
Suppose that $f \in R$ is integral over $J$.  Define $J' = J + (f )$ and let $Y' \to X$ be the blowup of $J'$.  Then $Y' \to X$ factors through $Y$ and $Y'$ is a partial normalization of $Y$ generated locally by adding a single integral element to the rings defining the affine charts $U_i$.
\end{lemma}
\begin{proof}
Write $f^n + a_1 f^{n-1} + \dots + a_n = 0$ with $a_i \in J^i$.  Now write $J = (z_1, \ldots, z_m )$ and form the Rees algebra $S$ as above.  Let $S' = R \oplus J' T \oplus J'^2 T \oplus \dots \supseteq S$.  We will first prove that the $U_i' = Y' \setminus V(z_i T)$ form an open cover of $Y'$ (in particular, we do not need $V(f T)$).  Suppose that $Q \subseteq S'$ is a homogeneous prime ideal containing all of the $z_i T$ but not $f T$.  Obviously $Q$ contains $0 = f^n T^n + a_1 f^{n-1} T^n + \dots + a_n T^n$ also note that $Q$ contains $a_n T^n$ since $a_n T^n \in ( z_1, \dots, z_n )^n T^n$.  But then since $Q$ does not contain $f T$, $Q$ must contain
\[
f^{n-1} T^{n-1} + a_1 f^{n-2} T^{n-1} + \dots + a_{n-1} T^{n-1}.
\]
But $Q$ also contains $a_{n-1} T^{n-1}$ as before and so continuing in this way, we eventually deduce that $f T \in Q$, a contradiction.  Thus we have shown that $\{U_i\}$ form an open cover of $\Proj S' = Y'$.

On the other hand, each $U_i' = \Spec R[z_1/z_i, \dots, z_m/z_i, f/z_i]$ and $y = f/z_i$ satisfies the monic polynomial equation
\[
(f/z_i)^n + (a_1/z_i) (f/z_i)^{n-1} + \dots + (a_n/z_i^n) = 0
\]
where each $a_j/z_i^j \in R[z_1/z_i, \dots, z_m/z_i]$ by construction.  The lemma follows.
\end{proof}

Next we recall a partial converse to the previous Lemma.

\begin{lemma}
\label{lem.BlowupOfNormalizedIdealPower}
Suppose additionally to \autoref{set.Blowup} that $R$ is normal, and that the normalization $\mu : Y' \to Y$ is finite over $Y$.  Then $\pi : Y' \to X$ is the blowup of $\overline{J^n}$ for some $n > 0$ where $\overline{\bullet}$ denotes the integral closure of the ideal.
\end{lemma}
\begin{proof}
Write $J = (z_1, \dots, z_m)$ and consider the ring $R_i := R[z_1/z_i, \dots, z_m/z_i]$ defining an affine chart $U_i$ on $Y$.  Suppose that $x \in \O_{Y'}(\mu^{-1} U_i)$, and hence $x$ is integral over $R_i$.  It follows that $x$ satisfies some integral equation
\[
x^l + f_{1} x^{l-1} + \dots + f_{l-1} x^1 + f_l = 0
\]
with $f_j = f_j(z_1/z_i, \dots, z_m/z_i) \in R_i$.  Note that we can pick a sufficiently large $h$ such that $f_jz_i^h\in J^h$ for all $j$ (i.e., clearing all the denominators of $f_j$). It follows that $f_j z_i^{hj} \in J^{hj} \subseteq R$ for all $j$. Multiplying by $z^{hl}$ we get
\[
(xz_i^h)^l + f_{1} z_i^h (xz_i^h)^{l-1} + \dots + f_{l-1} z_i^{h(l-1)} (x z_i^h)^1 + f_l z_i^{hl} = 0.
\]
Now, $x z_i^h$ is in $R$ since it is integral over $R$ and $R$ is normal.  Since $f_j z_i^{hj} \in J^{hj}$ for all $j$, we also have $xz_i^h \in \overline{J^h}$ and thus $x\in R[\frac{\overline{J^h}}{z_i^h}]$. We can do this for the finitely many generators of each chart, and pick $h\gg0$ that works for all these generators. It follows that there exists $h\gg0$ such that $\O_{Y'}(\mu^{-1} U_i)\subseteq R[\frac{\overline{J^h}}{z_i^h}]$ for every $i$. But then $\O_{Y'}(\mu^{-1} U_i)= R[\frac{\overline{J^h}}{z_i^h}]$ because the latter is integral over $R_i$ and $\O_{Y'}(\mu^{-1} U_i)$ is the integral closure of $R_i$. Therefore $Y'$ is the blow up of $\overline{J^h}$ as desired.
%It follows that $xz_i^{h+1} \in J$ and also that $f_j z_i^{(h+1)j} = J^j$.  In particular, $xz_i^{h+1}$
\end{proof}

\begin{remark}
Another way to prove this when $R$ is normal, Noetherian and excellent is to consider the Rees algebra $S$, and observe that the normalization $S'$ of $S$ is \[
S' = R \oplus \overline{J} T \oplus \overline{J^2} T^2 \oplus \dots,
\]
see for instance \cite[Proposition 5.2.1]{HunekeSwansonIntegralClosure}.  It easily follows that $\Proj S'$ is the normalization of $\Proj S$ \cite[6.C.9 Exercise]{PatilStorchIntroductiontoAGandCA}.  Since $S$ is excellent, $S'$ is finite over $S$ and hence Noetherian.  We thus see that $S'^{n}$, the $n$th Veronese of $S'$, is generated in degree $1$ for $n$ sufficiently divisible \cite[Chapter III, \S 1.3, Proposition 3]{Bourbaki1998}.  But $\Proj S'^{n} \cong \Proj S'$ is the blowup of $\overline{J^n}$.
\end{remark}

%\begin{corollary}
%\label{cor.BlowupRootsOK}
%Suppose that $J = \langle z_1, \dots, z_m \rangle$ and that $J' = \langle z_1^{1/d}, \dots, z_m^{1/d} \rangle \subseteq R$.  Then the blowup of $J'$ is a partial normalization of $J$ obtained by locally adjoining finitely many elements to the rings defining the affine charts $U_i$.
%\end{corollary}
%\begin{proof}
%Let $J_d = J'^d$.  By \autoref{lem.BlowupPowerAgreement}, the blowup of $J'$ is the same as the blowup of $J_d$.  On the other hand, each of the mixed terms of the monomial generators of $J_d$ is obviously integral over $J$ and so by \autoref{lem.IntegralElementPartialNormalization}, we are done.
%\end{proof}

%Finally, we also recall the following.

%\begin{lemma}
%\label{lem.ScaleBlowupIdealDoesntChangeBlowup}
%Suppose that $o$ is a nonzero divisor on $R$.  Then $Y$ the blowup of $J = \langle z_1, \dots, z_m \rangle$ and $Y'$, the blowup of $o \cdot J = \langle o z_1, \dots, o z_m \rangle$ agree. Furthermore the chart $U_i$ of $Y$ coincides with the chart corresponding to the complement of $V(o z_i T) \subseteq Y'$.
%\end{lemma}
%\begin{proof}
%Simply note that
%\[
%R[z_1/z_i, \dots, z_m/z_i] \cong R[oz_1/(oz_i), \dots, oz_m/(oz_i)].
%\]
%\end{proof}

Finally, we now move to blowups in Noetherian regular local rings.  First we recall some notation, suppose that $\pi : Y \to X = \Spec A$ is a finite type birational map between normal Noetherian integral schemes where $X$ is regular (or at least Gorenstein).  We also fix a choice of a dualizing complex $\omega_A^{\mydot}$ on $A$.  Since $A$ is Gorenstein and integral, this complex has cohomology only in a single degree (which we select to be $-\dim X$), and that cohomology is a line bundle which is denoted by $\omega_X$.  We then define the dualizing complex $\omega_Y^{\mydot}$ on $Y$ to be $\pi^! \omega_X^{\mydot}$ where we have sheafified our dualizing complex on $A$.  We also set $\omega_Y := \myH^{-\dim X} \omega_Y^{\mydot}$ and observe that this is not necessarily a line bundle.

By a \emph{canonical divisor} on $X$ we mean any Weil divisor $K_X$ on $X$ such that $\O_X(K_X) \cong \omega_X$.  Since $X$ is Gorenstein, $\O_X(K_X)$ is a line bundle and hence $K_X$ is Cartier.  Likewise a \emph{canonical divisor} on $Y$ is any Weil divisor $K_Y$ so that $\O_Y(K_Y) \cong \omega_Y$.

\begin{lemma}
\label{lem.ChoiceOfCanonical}
There exist canonical divisors $K_Y$ and $K_X$ that agree where $\pi$ is an isomorphism.  Furthermore, for any choice of $K_X$, there is such a compatible choice of $K_Y$.
\end{lemma}
Our proof also holds if $X$ is not necessarily Gorenstein but only normal with a dualizing complex.
\begin{proof}
First notice that even though $\omega_Y$ is not a line bundle, $\omega_Y$ is still a reflexive rank-1 sheaf, and so there exists a $K_Y$ with $\O_Y(K_Y) = \omega_Y$.  Consider the divisor $\pi_* K_Y$ on $X$ obtained by throwing away any irreducible component of $K_Y$ that is mapped to a subscheme of codimension $\geq 2$.  This divisor agrees with $K_Y$ wherever $\pi$ is an isomorphism, which is a set $U$ whose complement has codimension $\geq 2$ on $X$.  In particular, $\O_U(\pi_* K_Y) \cong \omega_X|_U$.  Thus $\O_X(\pi_* K_Y)$ is a reflexive sheaf that agrees with $\omega_X$ outside a set of codimension $\geq 2$, and so $\O_X(\pi_* K_Y) \cong \omega_X$, \cf \cite{HartshorneGeneralizedDivisorsOnGorensteinSchemes}.  Setting $K_X = \pi_* K_Y$ proves the first part of the lemma.

Now suppose that $K_X'$ is another choice of canonical divisor.  Since $\O_X(K_X') \cong \omega_X \cong \O_X(K_X)$, we see that $K_X' \sim K_X$ and so there exists some element $f$ of the fraction field $K(A)$ so that $K_X' = K_X + \Div_X(f)$.  We then set $K_Y' = K_Y + \Div_Y(f)$ and observe that $K_Y'$ and $K_X'$ agree where $\pi$ is an isomorphism.
\end{proof}

%We fix canonical divisors $K_Y$ and $K_X$ that agree on the locus where $\pi$ is an isomorphism (note a \emph{canonical divisor} on $Y$ is can be defined to be any divisor $K_Y$ on $Y$ such that $\O_Y(K_Y) \cong \myH^{-\dim X} \pi^! \omega_X[\dim X]$).  Note it is always possible to pick such canonical divisors, simply pick any such $K_Y$ and set $K_X := \pi_* K_Y$.  In you'd like to pick a different (linearly equivalent) canonical divisor on $X$, simply adjust $K_Y$ and $K_X$ by the principal divisor of the same rational function $f \in K(Y) = K(X)$.
\begin{definition}[Relative canonical divisor]
\label{def.RelCanonical}
Choose $K_Y$ and $K_X$ as in \autoref{lem.ChoiceOfCanonical}.  We define the \emph{relative canonical divisor} $K_{Y/X} := K_Y - \pi^* K_X$, and observe it is exceptional and also independent of the choice of $K_Y$ and $K_X$.  Note that if one chooses $\omega_X \cong \O_X$, then one may take $K_X = 0$ and so $K_Y = K_{Y/X}$ may be chosen to be exceptional.
\end{definition}

\begin{lemma}
\label{lem.RelativeCanonicalOfBlowup}
Suppose that $(R, \fram, k)$ is a regular local Noetherian ring of dimension $d$ and that $Y \to X = \Spec R$ is the blowup of $\fram$.  Then $Y$ is regular, has prime exceptional divisor $E$ with $\fram \O_Y  =\O_Y(-E)$ and $K_{Y/X} = (d-1)E$.
\end{lemma}
\begin{proof}
This is well known, but because we do not know of a reference where it is phrased in this language outside of the context of varieties over a field, we include a quick geometric proof.  Equivalent commutative algebra statements can be found for example in \cite{HerzogVasconcelosOnTheDivisorClassGroupOfRees,HerzogSimisVasconcelosCanModuleReesAlgebra,TomariWatanabeFilteredRings}.

A direct computation shows that the exceptional divisor $E \cong \bP^{d-1}_k$ lives in the regular scheme $Y$.  The same computation also shows that $\O_X(-E)|_E = \O_E(1)$.  Because we know\footnote{The adjunction formula still works in this generality, simply take the Grothendieck dual of the short exact sequence $0 \to \O_Y(-E) \to \O_Y \to \O_E \to 0$ which yields $0 \to \omega_Y \to \omega_Y(E) \to \omega_E \to 0$.} that $(K_Y + E)|_E = K_E$ and that $\O_E(K_E) = \O_E(-d)$, if we write $K_Y = nE$, then $(K_Y + E)|_E = (nE + E)|_E = K_E$ and so $-(n+1) = -d$ and thus $n = d-1$ as claimed.
\end{proof}

\bibliographystyle{skalpha}
\bibliography{MainBib}
\end{document}